\documentclass[12pt]{article}
\usepackage{amscd}\usepackage{amsmath}\usepackage{amssymb}
\usepackage{amsthm}\usepackage{latexsym}\usepackage{verbatim}
\usepackage{cite}
\newtheorem{definition}{Definition}

\newtheorem{theorem}{Theorem}[section]
\newtheorem{lemma}[theorem]{Lemma}
\newtheorem{prop}[theorem]{Proposition}

\newtheorem{corollary}[theorem]{Corollary}
\newtheorem{remark}{Remark}
\newcommand{\curl}{\text{curl}}
\newcommand{\pupx}{\frac{\partial{u}}{\partial{x}}}
\newcommand{\opeps}{\text{op}_\varepsilon}
\newcommand{\xiover}{\frac{\xi}{\sqrt{\varepsilon}}}
\newcommand{\supp}{\text{supp}}
\newcommand{\Z}{\mathbb{Z}}
\newcommand{\R}{\mathbb{R}}
\newcommand{\T}{\mathbb{T}}
\newcommand{\N}{\mathbb{N}}\newcommand{\piperp}{\mathbb{P}_{\xi_o^{\perp}}}

\begin{document}
\title{Linear instability criteria for ideal fluid flows subject to two subclasses of perturbations}
\author{Elizabeth Thoren}
\maketitle

\section{Introduction}

The criteria for linear instability can be reduced to conditions on the spectral radius of the linear evolution operator.  More specifically, we can demonstrate the instability of some flow if we find the spectral radius of the associated linear evolution operator to be greater than 1 for some positive time, $t$.   The approach here involves computing the radius of a subset of the spectrum known as the essential spectrum.  This quantity is equal to a Lyapunov-type exponent associated with the equilibrium flow, see \cite{vishik1, fv1, fv2, fv3, fsv} for example.

For this paper we examine the linear stability of a smooth steady ideal fluid flow on a periodic fluid domain and consider the linear evolution of two separate classes of perturbations arising naturally from the group structure of hydrodynamics.   Our first space of perturbations is the tangent space to the orbit of the steady flow under the co-adjoint action of the group.  The co-adjoint orbit of a divergence free vector field is the collection of isovorticial fields, so these linear perturbations infinitesimally preserve the topology of the vorticity of the equilibrium flow.  We define these perturbations to be the closure of the image of a certain linear operator $B$, so we denote the space $\overline{\text{Im}B}$.  We also consider the linearized flow on the canonical factor space $L^2_{sol}/\overline{\text{Im}B}$.

The main results of this paper are extensions of a method developed by Vishik in \cite{vishik1} for computing the essential spectral radius of the evolution operator $G(t)$ associated with a smooth periodic solution, $u$, to steady Euler's equation.  To compute this quantity for our evolution operator, we must define a Lyapunov-type exponent associated with the following bicharacteristic amplitude system:
\begin{equation}\renewcommand{\arraystretch}{1.5}(\text{BAS})\left\{\begin{array}{l}
\dot{x}= u(x),  \\
\dot{\xi} = -\Bigl( \frac{\partial{u}}{\partial{x}}\Bigr)^T\xi, \\
\label{b}
\dot{b}= -\Bigl(\pupx \Bigr)b + 2\Bigl( \pupx b, \xi \Bigr) \frac{\xi}{|\xi|^2},\\
\bigl(x(0),\xi(0),b(0)\bigr) = (x_0,\xi_0,b_0) \in \mathcal{A},\end{array}\right.
\end{equation}
where the set of admissible initial conditions $\mathcal{A}$ is defined by
\begin{equation*}
\mathcal{A}:= \{(x_0,\xi_0,b_0)\in \T^n \times \R^n \times \R^n |\ \xi_0 \perp b_0 ,\ |\xi_0|=|b_0|=1 \}.
 \end{equation*}
 
\begin{theorem}[Vishik `96]
\label{vishik}
Let  $\mu$ be the following Lyapunov-type exponent:
\begin{equation*}
\mu := \lim_{t\to\infty} \frac{1}{t}\log\sup_{(x_0,\xi_0,b_0)\in \mathcal{A}}|b(x_0,\xi_0,b_0;t)|,
\end{equation*}
then $r_{ess}(G(t))= e^{\mu t}$.
\end{theorem}

For a given 2- or 3-dimensional smooth steady periodic fluid flow, we establish lower bounds for the radius of the essential spectrum of the linear evolution operator on each class of perturbations in terms of a series of Lyapunov-type exponents based on the bicharacteristic amplitude system above.

  For a 3-dimensional fluid flow $u$, let $\omega := \curl u$ be the vorticity of our steady flow and define the following Lyapunov-type exponents:
\begin{align*}
 \mu_{3*} = \lim_{t\to\infty}\frac{1}{t}\log\sup_{\substack {(x_0,\xi_0,b_0)\in \mathcal{A}\\x_0 \in \supp(\omega)}} |b(x_0, \xi_0, b_0; t)|\\
\mu_{3F} = \lim_{t\to\infty}\frac{1}{t}\log\sup_{\substack {(x_0,\xi_0,b_0)\in \mathcal{A}\\x_0 \notin \supp(\omega)}} |b(x_0, \xi_0, b_0; t)|,
\end{align*}
where $b(x_0, \xi_0, b_0; t)$ denotes a solution to (BAS) at time $t>0$ with initial conditions $(x_0,\xi_0,b_0)$.  Then we have the following lower bound for the essential spectral radius of the linear evolution operator restricted to perturbations tangent to the co-adjoint orbit of $u$, $\overline{\text{Im}B}$:  $r_{ess}(G(t)|_{\overline{\text{Im}B}}) \geq e^{\mu_{3*}t}$.  And we have another lower bound for the essential spectral radius of the linear evolution acting on the factor space: $r_{ess}(G_F(t)) \geq e^{\mu_{3F}t}$.

For 2-dimensional flows, our classes of perturbations are described in terms of the scalar vorticity and the resulting exponents depend on its gradient, $\nabla\omega$.  Define the Lyapunov-type exponent $\mu_{2*}$ by
\begin{equation*}
 \mu_{2*} = \lim_{t\to\infty}\frac{1}{t}\log\sup_{\substack {(x_0,\xi_0,b_0)\in \mathcal{A} \\x_0 \in \supp (\nabla \omega)}} |b(x_0, \xi_0, b_0; t)|.
\end{equation*} And define $\mu_{2F}$ by
\begin{equation*}
\mu_{2F} = \lim_{t\to\infty}\frac{1}{t}\log \sup_{(x_0,\xi_0,b_0)\in \mathcal{A}_1\cup \mathcal{A}_2} |b(x_0, \xi_0, b_0; t)|,
\end{equation*} where 
\begin{align*}\mathcal{A}_1:=& \{(x_0,\xi_0,b_0)\in \mathcal{A}: x_0 \notin \supp(\nabla \omega)\},\\
\mathcal{A}_2:=& \{(x_0,\xi_0,b_0)\in \mathcal{A}: \nabla\omega(x_0)\neq 0,\ b_0\perp\nabla\omega(x_0)\}.
\end{align*}
Then we have similar lower bounds for the essential spectral radius of the linear evolution on each class of perturbations:
$r_{ess}(G(t)|_{\overline{\text{Im}B}}) \geq e^{\mu_{2*}t}$ and $r_{ess}(G_F(t)) \geq e^{\mu_{2F}t}$.

We make use of Vishik's approximation of $G(t)$ on high-frequencies by a pseudodifferential operator composed with parallel transport along the flow.  Section \ref{approxG(t)} details this approximation and its connection with (BAS) along with several related lemmas necessary for proofs of the main theorems.  In Sections \ref{3Dclassifysection} and \ref{2Dclassifysection} we introduce sufficient criteria for high frequency perturbations to approximately be in each of our two classes.    The main theorems of this paper, Theorem \ref{3Dmainthm} and Theorem \ref{2Dmainthm}, are proved in sections \ref{mainthms3D} and \ref{mainthms2D}.  In the last section we discuss instability near a hyperbolic stagnation point in the context of our two classes of perturbations.

\section{Two classes of perturbations}
\label{twoclasses}
For this paper we let $u\in C^{\infty}(\T^n)$ be a smooth solution to steady Euler's equation on the 2- or 3-dimensional torus $\T^n = \R^n/\Z^n$ and consider the linear evolution operator associated with Euler's equation linearized at $u$
\begin{equation*}\bigl(\text{LE}\bigr)\left\{\begin{array}{l}
\partial_t w = -u\cdot \nabla w - w\cdot \nabla u -\nabla q\\
w(x,0)= w_0(x)\end{array}\right.
\end{equation*}
where $\nabla q \in L^2(\T^n)$ is the gradient of a scalar pressure uniquely determined by the requirement that solutions $w(x,t)$ remain divergence free and our initial perturbation $w_0(x)$ is a small divergence free square integrable vector field on $\T^n$.  We let $L^2_{sol}$ denote the space of divergence free square integrable vector fields on $\T^n$ and define $G(t): L^2_{sol} \rightarrow L^2_{sol}$ to be the solution operator to the linearized system.  In other words, $w(x,t):= G(t)w_0(x)$ is a unique solution to (LE).

Our first class of perturbations are those that preserve the topology of vortex lines and we define them in terms of the operator $B:L^2_{sol} \to L^2_{sol}$, given by
\begin{equation*}
 Bv := \omega \times v - \nabla\alpha,
\end{equation*}
where $\omega := \curl u$ is the vorticity and the pressure gradient $\nabla \alpha$ is uniquely determined by the requirement that $Bv$ is divergence free.  The operator $B$ is actually a representation of the co-adjoint action of the Lie algebra of divergence free vector fields acting on our steady flow $u$, hence its image is the tangent space to the co-adjoint orbit of $u$.  From the Hodge decomposition, we have an equivalent formulation for $B$ in terms of $\mathbb{P}_{sol}$, the projection onto $L^2_{sol}$:
\begin{equation*}
Bv = \mathbb{P}_{sol}(\omega \times v).
\end{equation*}
 In particular the flow of an Eulerian fluid will stay within the co-adjoint orbit of the initial condition.    A straightforward computation demonstrates that $\overline{\text{Im}B}$ is an invariant subspace under the linearized flow, so the essential spectral radius of the evolution of perturbations in $\overline{\text{Im}B}$ is well defined.  We also consider the linearized flow on the factor space $L^2_{sol}/\overline{\text{Im}B}$ with the canonical factor space norm.  This factor space forms our second class of perturbations.  For a thorough discussion of the group structure of hydrodynamics see Arnold and Khesin's book \cite{a&k}. 

\section{3-dimensional high frequency vector fields}
\label{3Dclassifysection}

This section contains several lemmas regarding high frequency vector fields to be used in computing lower bounds for the essential spectral radius of the linear evolution operator acting on each class of perturbations.  The goal is to establish criteria for these perturbations so they approximate perturbations in $\overline{\text{Im}B}$ or that we may estimate their growth in the factor space $F := L^2_{sol}/\overline{\text{Im}B}$.

\begin{lemma}
\label{solproj}
Let $v$ be a vector field in $H^1(\mathbb{T} ^n)$, $\xi_0\in \Z^n$ for $n = 2,3$ and $\delta^{-1}\in \Z_+$.  Then
\begin{equation*}
\|\mathbb{P}_{sol} (v(x)e^{ix\cdot\xi_0/\delta}) - \piperp (v(x))e^{ix\cdot\xi_0/\delta}\|_{L^2} \leq \delta \frac{C}{|\xi_0|} \|v\|_{H^1},
\end{equation*}
where $\mathbb{P}_{sol}$ denotes the orthogonal projection of $L^2$ onto $L^2_{sol}$.
\end{lemma}
 \begin{proof}

We may treat the 2-dimensional case as planar vector fields in 3-dimensions, thus it suffices to prove the lemma in 3-dimensions.  Assume $v\in \bigl(H^1(\T^3)\bigr)^3$.  Define a vector field, $\alpha\in L^2_{sol}(\T^3)$, that approximates the projection of $v(x)e^{ix\cdot\xi_0/\delta}$ onto $\xi_0^{\perp}$:
\begin{equation*}
\alpha(x) := \delta \nabla\times \Bigl( \frac{i\xi_0 \times v(x)}{|\xi_0|^2}e^{ix\cdot\xi_0/\delta}\Bigr)
\end{equation*}
Since $(\xi_0 \times \piperp(v))\times \xi_0 = \piperp(v)$ we have
\begin{equation*}
\alpha(x)= \piperp(v(x))e^{ix\cdot\xi_0/\delta} + \delta\Bigl[\Bigl(\nabla\times \frac{i\xi_0 \times \piperp(v(x))}{|\xi_0|^2}\Bigr)e^{ix\cdot\xi_0/\delta}\Bigr].
\end{equation*}
It follows that $\|\alpha - \piperp(v)e^{i(\cdot) \cdot \xi_0/\delta}\|_{L^2} \leq \delta \frac{1}{|\xi_0|} \|v\|_{H^1}$.

Now we define an gradient vector field, $\beta \in L^2_{grad}(\T^3)$, that approximates the projection of $v(x)e^{ix\cdot\xi_0/\delta}$ in the direction of $\xi_0$:
\begin{align*}
\beta(x) :=& -\frac{i\delta}{|\xi_0|^2} \nabla ((\xi_0, v(x))e^{ix\cdot\xi_0/\delta})
\\ =& \mathbb{P}_{\xi_0}(v(x))e^{ix\cdot\xi_0/\delta} -\frac{i\delta}{|\xi_0|^2} \nabla (\xi_0, v(x))e^{ix\cdot\xi_0/\delta}.
\end{align*}
Thus, $\|\beta - \mathbb{P}_{\xi_0}(v)e^{i(\cdot)\cdot\xi_0/\delta}\|_{L^2} \leq \delta \frac{1}{|\xi_0|} \|v\|_{H^1}$.  From the Hodge decomposition we know $L^2(\T^3) = L^2_{sol}(\T^3) \oplus L^2_{grad}(\T^3)$ and, from the computations above, $\piperp(v)e^{i(\cdot) \cdot \xi_0/\delta}$ is approximately soleniodal while $\mathbb{P}_{\xi_0}(v)e^{i(\cdot)\cdot\xi_0/\delta}$ is approximately a gradient.  It follows that
\begin{equation}
\label{3Dversion}
 \|\mathbb{P}_{sol} (v(x)e^{ix\cdot\xi_0/\delta}) - \piperp (v(x))e^{ix\cdot\xi_0/\delta} \|_{L^2(\T^3)} \leq \delta \frac{C}{|\xi_0|} \|v\|_{H^1(\T^3)}.
\end{equation}
\end{proof}


Here we define the basic structure of our fast oscillating vector fields.  In Section \ref{2Dclassifysection} we discuss the special case of 2-dimensional fast oscillating vector fields, but here we are working in 3-dimensions.  Define $\psi_{\delta} \in \bigl(L^2_{sol}(\T^3)\bigr)^3$ by
\begin{equation}
\label{definepsi}
\psi_{\delta}(x) = \delta \nabla \times \left(\frac{i\xi_0 \times P}{|\xi_0|^2}h_0(x)e^{ix\cdot\xi_0/\delta}\right),
\end{equation}
where $\xi_0 \in \Z^3, \delta ^{-1} \in \Z_+, P \perp \xi_0$ is a constant vector and $h_0 \in C^{\infty}(\T^3)$ is an arbitrary smooth scalar function.  Notice that we can expand the expression for $\psi_{\delta}$ to get
\begin{equation}
\label{psiexpanded}
 \psi_{\delta}(x) = h_0(x)Pe^{ix\cdot\xi_0/\delta} + \delta\bigr[\nabla h_0(x) \times \Bigr( \frac{i\xi_0\times P}{|\xi_0|^2} \Bigl)e^{ix\cdot\xi_0/\delta}\bigl].
\end{equation}

This next lemma gives criteria for these fast oscillating vector fields to be close to $\overline{\text{Im}B}$ in 3-dimensions.  The criteria requires that we introduce a parameter $\zeta$ that localizes the support of the fast oscillating vector field.

\begin{lemma}
\label{inimage}
 Let $x_0 \in \T^3$, $\xi_0 \in \Z^3$ such that $(\omega(x_0), \xi_0) \neq 0$.  Let $h_0\in C^{\infty}(\T^3)$ such that $\supp h_0 \subset B_1(0)$, the ball centered at $0$ of radius $1$.  Let $0 < \zeta < 1$ and define $h_{\zeta}$ by
\begin{equation*}
 h_{\zeta} := h_0\Bigl(\frac{x-x_0}{\zeta}\Bigr).
\end{equation*}
And let
\begin{equation*}
 \psi_{\zeta,\delta}(x):= \delta \nabla \times \left(\frac{i\xi_0 \times P}{|\xi_0|^2}h_{\zeta}(x)e^{ix\cdot\xi_0/\delta}\right),
\end{equation*}
where $P \perp \xi_0$ is a constant vector and $\delta^{-1} \in \Z_+$.
Then there exists $\overline{\psi}_{\zeta,\delta} \in L^2_{sol}$ such that
\begin{equation*}
 \psi_{\zeta,\delta} - B(\overline{\psi}_{\zeta,\delta}) = r_{\zeta} + r_{\delta},
\end{equation*}
where $\|r_{\zeta}\|_{L^2} \leq c_0\zeta^{5/2}$ for some constant $c_0 >0$ that does not depend on $\delta$ and $\|r_{\delta}\|_{L^2} = O(\delta)$.
\end{lemma}

\begin{proof}
 First we find an appropriate constant vector $Q\perp \xi_0$ to play the role of $P$ in our preimage $\overline{\psi}_{\zeta,\delta}$.  Let $T:\piperp (\R^3) \to \piperp (\R^3)$ be defined by
\begin{equation*}
 Tv := \piperp (\omega(x_0) \times v).
\end{equation*}
Our assumption that $(\omega(x_0),\xi_0)\neq 0$ implies that $T$ is a bijection on $\piperp (\R^3)$.   Hence there is a constant vector $Q \perp \xi_0$ such that $P = \piperp (\omega(x_0)\times Q)$.

Define the vector field $\overline{\psi}_{\zeta,\delta} \in C^{\infty}_{sol}(\T^3)$ by
\begin{equation*}
 \overline{\psi}_{\zeta,\delta}(x):= \delta \nabla \times \left(\frac{i\xi_0 \times Q}{|\xi_0|^2}h_{\zeta}(x)e^{ix\cdot\xi_0/\delta}\right).
\end{equation*}
Then from the expansion (\ref{psiexpanded}) and the linearity of $B$ we have
\begin{equation*}
 B(\overline{\psi}_{\zeta,\delta}) = B(h_{\zeta}Qe^{i(\cdot)\cdot\xi_0}) + \delta B\bigr[\nabla h_{\zeta} \times \Bigr( \frac{i\xi_0\times Q}{|\xi_0|^2} \Bigl)e^{ix\cdot\xi_0/\delta}\bigl].
\end{equation*}
We may also expand $\psi_{\zeta,\delta}$ as in (\ref{psiexpanded}) to get
\begin{equation}
\label{estimate}
 \psi_{\zeta,\delta} - B(\overline{\psi}_{\zeta,\delta}) = h_{\zeta}Pe^{i(\cdot)\cdot\xi_0/\delta} - B(h_{\zeta}Qe^{i(\cdot)\cdot\xi_0}) + \delta R_1,
\end{equation}
where
\begin{equation*}
R_1= \bigr[\nabla h_{\zeta} \times \Bigr( \frac{i\xi_0\times P}{|\xi_0|^2} \Bigl)\bigl] - B\bigr[\nabla h_{\zeta} \times \Bigr( \frac{i\xi_0\times Q}{|\xi_0|^2} \Bigl)e^{ix\cdot\xi_0/\delta}\bigl].
\end{equation*}
  Hence, $\|R_1\|_{L^2} \leq \Bigl(\frac{|P|}{|\xi_0|}\|h_{\zeta}\|_{H^1} + \frac{|Q|}{|\xi_0|}\|B\|_{\mathcal{L}(L^2)}\|h_{\zeta}\|_{H^1}\Bigr)$.  Notice that $\|B\|_{\mathcal{L}(L^2)} \leq \|\omega\|_{L^{\infty}}$, so
\begin{equation}
\label{R_1}
\|R_1\|_{L^2} \leq  \Bigl(\frac{|P|}{|\xi_0|} + \frac{|Q|}{|\xi_0|}\|\omega\|_{L^{\infty}}\Bigr)\|h_{\zeta}\|_{H^1}.
\end{equation}

To get a bound on the main order term of the RHS of (\ref{estimate}) we first use Lemma \ref{solproj} to compute
\begin{align*}
 B(h_{\zeta}Qe^{i(\cdot)\cdot\xi_0/\delta}) &:= \mathbb{P}_{sol}(\omega \times h_{\zeta}Qe^{i(\cdot)\cdot\xi_0/\delta})\\
& = h_{\zeta}\piperp(\omega\times Q)e^{i(\cdot)\cdot\xi_0/\delta} + R_{\delta},
\end{align*}
where
\begin{equation}
\label{rdelta}
\|R_{\delta}\|_{L^2} \leq \delta \frac{C}{|\xi_0|}\|h_{\zeta}\omega\|_{H^1}.
 \end{equation}
 Define
\begin{equation*}
r_{\zeta}:= h_{\zeta}Pe^{i(\cdot)\cdot\xi_0/\delta} - h_{\zeta}\piperp(\omega\times Q)e^{i(\cdot)\cdot\xi_0/\delta}.
 \end{equation*}
Then we may write the main order term from the RHS of (\ref{estimate}) as
\begin{equation}
\label{otherpart}
 h_{\zeta}Pe^{i(\cdot)\cdot\xi_0/\delta} - B(h_{\zeta}Qe^{i(\cdot)\cdot\xi_0/\delta}) = r_{\zeta} + R_{\delta}.
\end{equation}
We will demonstrate that $\|r_{\zeta}\|_{L^2} \leq c_0\zeta^{5/2}$ where the constant $c_0$ is positive and does not depend on $\delta$.
Since $P = \piperp (\omega(x_0)\times Q)$ and $\supp h_{\zeta}$ is contained in the ball of radius $\zeta$ centered at $x_0$, $B_{\zeta}(x_0)$, we have
\begin{align}
\nonumber
\|r_{\zeta}\|_{L^2} &= \|h_{\zeta}Pe^{i(\cdot)\cdot\xi_0/\delta} - h_{\zeta}\piperp(\omega\times Q)e^{i(\cdot)\cdot\xi_0/\delta}\|_{L^2}\\
 \label{zetabound}&\leq \|h_{\zeta}\|_{L^2}\|\piperp \bigl(\omega(x_0) - \omega(\cdot)) \times Q)\bigr)\|_{L^{\infty}(B_{\zeta}(x_0))}.
\end{align}
Since $\omega(x)$ is Lipschitz and for any $x\in \supp h_{\zeta}$, $|x-x_0| \leq \zeta$, it follows that
\begin{equation*}
\|\piperp \bigl(\omega(x_0) - \omega(\cdot)) \times Q)\bigr)\|_{L^{\infty}(B_{\zeta}(x_0))} \leq \zeta \|\omega\|_{Lip}|Q|.
\end{equation*}
And since $\|h_{\zeta}\|_{L^2} = \zeta^{3/2}\|h_0\|_{L^2}$, we have from estimate (\ref{zetabound}) that
\begin{equation*}
\|r_{\zeta}\|_{L^2}\leq \zeta^{5/2}\|h_0\|_{L^2}\|\omega\|_{Lip}|Q|.
\end{equation*}
 Let $c_0 = \|h_0\|_{L^2}\|\omega\|_{Lip}|Q|$, which is independent of $\delta$.  Now define $r_{\delta} := \delta R_1 + R_{\delta}$.    Therefore, from (\ref{estimate}) and (\ref{otherpart}) we have
\begin{equation*}
 \psi_{\zeta,\delta} - B(\overline{\psi}_{\zeta,\delta}) = r_{\zeta} + r_{\delta}.
\end{equation*}
From (\ref{R_1}) and (\ref{rdelta}) we have $\|r_{\delta}\|_{L^2} = O(\delta)$.
\end{proof}

\begin{remark}
For fast oscillating vector fields like $\psi_{\delta}$ in 2-dimensions, $(\omega,\xi_0)\equiv 0$, so this lemma does not give us any information about $\overline{\text{Im}B}$ in 2-dimensions.
\end{remark}

\section{2-dimensional high frequency vector fields}
\label{2Dclassifysection}
We treat 2-dimensional flows as 3-dimensional planar flows to get a simplified form of the operator $B$ in 2-dimensions.
Thus
\begin{equation*}
 Bv:= \omega \times v - \nabla \alpha,
\end{equation*}
can be simplified to
\begin{equation}
\label{2DB}
 Bv = \omega \cdot v^{\perp} - \nabla \alpha = \mathbb{P}_{sol}(\omega \cdot v^{\perp}),
\end{equation}
when $v\in (L^2_{sol}(\T^2))^2$.  In (\ref{2DB}) $\omega$ is now the scalar vorticity.  The pressure $\nabla \alpha \in (L^2(\T^2))^2$ is determined by the requirement that $Bv$ be divergence free  and $\mathbb{P}_{sol}$ is the orthogonal projection onto divergence free vector fields.

Let $\phi_{\delta} \in (L^2_{sol}(\T^2))^2$ be defined by,
\begin{equation}
 \label{phidelta}
\phi_{\delta}(x):= -i\delta \nabla^{\perp}(h_0(x)e^{ix\cdot\xi_0/\delta}),
\end{equation}
where $\xi_0 \in \Z^2, \delta ^{-1} \in \Z_+, P \perp \xi_0$ is constant and $h_0 \in C^{\infty}(\T^2)$ is an arbitrary smooth scalar function.
We can expand $\phi_{\delta}$ as follows:
\begin{equation}
\label{phiexpanded}
 \phi_{\delta}(x) = h_0(x)\xi_0^{\perp}e^{ix\cdot\xi_0/\delta} -i\delta\bigl[e^{ix\cdot\xi_0/\delta}\nabla^{\perp}h_0(x)\bigr].
\end{equation}

In this next Lemma we establish criteria for $\phi_{\delta}$ to be near $\overline{\text{Im}B}$.  Our criteria is based on $\nabla \omega$, the gradient of the scalar vorticity of the equilibrium solution $u$.

\begin{lemma}
\label{phideltainIm}
Define $\phi_{\delta}$ as in (\ref{phidelta}) above.  If there is a constant $c_0$ such that $|(\xi_0^{\perp},\nabla \omega (x))|> c_0$ on $\supp h_0$, then there exists a remainder $r_{\delta}\in L^2$ such that $\phi_{\delta} + r_{\delta} \in \overline{\text{Im}B}$ and $\|r_{\delta}\|_{L^2}= O(\delta)$.
\end{lemma}
\begin{proof}
Assume there exists a constant $c_0$ such that $|(\xi_0^{\perp},\nabla \omega (x))|> c_0$ on $\supp h_0$.  Then we can define a function $g_0 \in C^{\infty}(\T^2)$ by
\begin{equation}
\label{g0}
g_0(x):= \frac{|\xi_0|^2h_0(x)}{(\xi_0^{\perp},\nabla \omega(x))},
\end{equation}
and define a vector field $v \in C^{\infty}_{sol}(\T^2)$ by,
\begin{equation*}
v(x):= \nabla^{\perp}(g_0(x)e^{ix\cdot\xi_0/\delta}).
\end{equation*}
From (\ref{2DB}) the operator $B$ on $\nabla^{\perp}(g_0(x)e^{ix\cdot\xi_0/\delta})$ takes this simplified form:
\begin{align*}
Bv &= \mathbb{P}_{sol}\bigl(\omega \nabla(g_0(x)e^{ix\cdot\xi_0/\delta})\bigr)\\
&= \mathbb{P}_{sol}\bigl(\nabla (\omega g_0e^{ix\cdot\xi_0/\delta})\bigr) - \mathbb{P}_{sol}\bigl(g_0(x)e^{ix\cdot\xi_0/\delta}\nabla \omega\bigr)\\
&= - \mathbb{P}_{sol}(g_0(x)e^{ix\cdot\xi_0/\delta}\nabla \omega),
\end{align*} since the gradient of a function is irrotational and, hence, orthogonal to the space of divergence free vector fields.  If we apply Lemma \ref{solproj} we have
\begin{equation*}
Bv = -g_0(x)\mathbb{P}_{\xi_0^{\perp}}(\nabla\omega)e^{ix\cdot\xi_0/\delta} + \tilde{r}_{\delta},
\end{equation*}
where $\|\tilde{r}_{\delta}\|_{L^2} = O(\delta)$.  Substitute our definition for $g_0$ from (\ref{g0}) to get
\begin{equation*}
Bv = \xi_0^{\perp}h_0(x)e^{ix\cdot\xi_0/\delta} + \tilde{r}_{\delta}.
\end{equation*}
Then the expansion (\ref{phiexpanded}) for $\phi_{\delta}$ implies
\begin{equation*}
Bv = -i\delta\nabla^{\perp}(h_0(x)e^{ix\cdot\xi_0/\delta}) + r_{\delta} =: \phi_{\delta} + r_{\delta},
\end{equation*}
where $r_{\delta}:= \tilde{r}_{\delta} -i\delta\bigl[e^{ix\cdot\xi_0/\delta}\nabla^{\perp}h_0(x)\bigr]\in L^2$ and $\|r_{\delta}\|_{L^2} = O(\delta)$.  Thus we have $\phi_{\delta} + r_{\delta} \in \overline{\text{Im}B}$.
\end{proof}

This lemma establishes criteria for measuring the factor space norm of our fast oscillating vector fields.  We must introduce the parameter $\zeta$ to localize the support of these fields.  Recall that our factor space $F:= L^2_{sol}(\T^2)/\overline{\text{Im}B}$, with the canonical factor space norm we denote $\|\cdot\|_F$.

\begin{lemma}
\label{phideltainKer}
Let $x_0 \in \T^n$, $\xi_0\in \R^n$ such that $\nabla \omega (x_0) \neq 0$ and $(\xi_0^{\perp},\nabla \omega(x_0))= 0$.  Let $h_0\in C^{\infty}(\T^n)$ be supported on $B_1(0)$, the ball of radius $1$ centered at $0$ such that $h_0(0)=1$.  For $0<\zeta<<1$ define $h_{\zeta}$ by
\begin{equation*}
h_{\zeta}(x):= h_0\bigl(\frac{x-x_0}{\zeta}\bigr),
\end{equation*}
and let $\delta^{-1}\in \Z_+$.  For any $x \in [0,1)\times [0,1)$ define
\begin{equation*}
\phi_{\zeta,\delta}(x):= -i\delta \nabla^{\perp}(h_{\zeta}(x)e^{ix\cdot\xi_0/\delta}),
\end{equation*}
and extend $\phi_{\zeta,\delta}$ periodically.  Then we have
\begin{equation*}
 \|\phi_{\zeta,\delta}\|_F = \|\phi_{\zeta,\delta}\|_{L^2} + O(\zeta) + O(\delta),
\end{equation*}
where $O(\zeta)$ is independent of $\delta$ and $O(\delta)$ is independent of $\zeta$.
\end{lemma}
\begin{remark}
\label{zetarmk}
The conditions on $x_0$ and $\xi_0$ imply that $\xi_0$ is a scalar multiple of $\nabla \omega(x_0)$, so we cannot require $\xi_0\in \Z^2$ here.  To ensure that $\phi_{\zeta,\delta}$ is periodic, we define the vector field on $B_{\zeta}(x_0)$ and, since $\zeta << 1$, we may extend it periodically.
\end{remark}

\begin{proof}
Since $B$ maps into $L^2_{sol}$, we can say $v \in \text{Ker}B$ if and only if $v \in \text{Ker}T$ where $T:(L^2_{sol}(\T^2))^2 \to (L^2_{sol}(\T^2))^2$ is defined by
\begin{equation*}
 Tv := \curl Bv = v\cdot \nabla \omega.
\end{equation*}
For any $x\in \supp (h_{\zeta})$, $|x-x_0| \leq \zeta$, so
\begin{equation*}
|\nabla \omega (x) - \nabla \omega (x_0)| \leq \zeta K,
\end{equation*}
where $K := \|\nabla\omega\|_{Lip}$ is the Lipschitz norm of $\nabla\omega$.  We may assume $\zeta << |\nabla\omega(x_0)|$, so for any $x\in \supp h_{\zeta}$
\begin{equation*}
|\nabla\omega (x)| \geq |\nabla\omega(x_0)| - \zeta K> 0.
\end{equation*}
We assume $(\xi_0^{\perp},\nabla\omega(x_0)) = 0$, so we have
\begin{equation*}
\frac{|(\xi_0^{\perp},\nabla\omega(x))|}{|\nabla\omega(x)|} = \frac{|(\xi_0^{\perp},\nabla\omega(x))-(\xi_0^{\perp},\nabla\omega(x_0)) |}{|\nabla\omega(x)|}
\leq \frac{\zeta|\xi_0^{\perp}|K}{|\nabla\omega(x_0)|-\zeta K}.
\end{equation*}
For any $x \in \supp h_{\zeta}$, let
\begin{equation}
\label{eta}
\eta(x):= \ \xi_0^{\perp} - \frac{(\xi_0^{\perp},\nabla\omega(x))}{|\nabla\omega(x)|^2}\nabla\omega(x)
  = \ \xi_0^{\perp} - O(\zeta)\frac{\nabla\omega(x)}{|\nabla\omega(x)|}.
\end{equation}
We can expand $\phi_{\zeta,\delta}$ as in (\ref{phiexpanded}) and compute
\begin{equation*}
 \phi_{\zeta, \delta}(x) = h_{\zeta}(x)\xi_0^{\perp}e^{ix\cdot\xi_0/\delta} + r_{\delta},
\end{equation*}
where $\|r_{\delta}\|_{L^2} \leq \delta C\|\nabla h_{\zeta}\|_{L^2}$.  Notice that in 2-dimensions, $\|\nabla h_{\zeta}\|_{L^2} = \|\nabla h_0\|_{L^2}$, so $\|r_{\delta}\|_{L^2}\leq \delta C \|\nabla h_0\|_{L^2}$, which is independent of $\zeta$. We also have from the definition of $\eta$ in (\ref{eta}) that
\begin{equation*}
\phi_{\zeta, \delta}(x) = h_{\zeta}(x)\eta(x)e^{ix\cdot\xi_0/\delta} + r_{\zeta}+ r_{\delta},
\end{equation*}
where $\|r_{\zeta}\|_{L^2} = O(\zeta)$ independent of $\delta$.
Since $(T\eta)(x):= \eta(x)\cdot \nabla\omega(x) \equiv 0$, we have $\phi_{\zeta,\delta} - r_{\zeta} - r_{\delta} \in \text{Ker}B$.  Therefore, $\|\phi_{\zeta,\delta}\|_F = \|\phi_{\zeta,\delta}\|_{L^2}+ O(\zeta) + O(\delta)$.

\end{proof}


\section{Approximating $G(t)$ on high frequency vector fields}
\label{approxG(t)}

This section details the method of approximating $G(t)$ on high frequencies by a pseudodifferential operator and some immediate consequences of that approximation.  We also derive the bicharacteristic amplitude system (BAS) in terms of the symbol of this pseudodifferential operator.   We finish the section with several lemmas to be used in proving the main theorems.  Lemmas \ref{sigmaS-1} gives us estimates for the norm of certain pseudodifferential operators and Lemmas \ref{imagepsilemma} and \ref{imagephilemma} motivate the structure of our high frequency vector fields.

We introduce an $\varepsilon$-psuedodifferential operator to separate vector fields into their high- and low-frequency parts.
Let $\varepsilon >0$, for any amplitude $\sigma \in C^{\infty}(\T^n \times \R^n)$ (satisfying appropriate conditions to be specified later) define
\begin{equation}
\label{opepsdef}
 (\opeps[\sigma]w)(x) :=\ \frac{1}{(2\pi\varepsilon)^n}\int  \sigma(x,\xi) e^{i\xi\cdot(x-y)/\varepsilon}w(y)\ dyd\xi.
\end{equation}
Let $\chi(\xi) \in C^{\infty}(\R^n)$ be a function of $|\xi|$ only, with $0 \leq \chi(\xi) \leq 1$, and
\begin{equation*}
\chi (\xi) = \begin{cases}
                 1& \text{if } |\xi| \leq \frac{1}{2},\\
                 0& \text{if } |\xi| \geq \frac{2}{3}.\end{cases}
\end{equation*}
Then
\begin{equation*}
 G(t) = G(t)\circ \opeps\Bigl[1-\chi\big(\frac{\xi}{\sqrt{\varepsilon}}\bigr)\Bigr] + G(t)\circ \opeps\Bigl[\chi\big(\frac{\xi}{\sqrt{\varepsilon}}\bigr)\Bigr].
\end{equation*}  

Our focus on high frequency vector fields is a consequence of Nussbaum's formula for computing the essential spectral radius of a bounded linear operator, so we define the essential spectrum and state the formula here.  A proof can be found in Nussbaum's original paper, \cite{nussbaum}.  

We may introduce the following classification of points in the spectrum of a bounded linear operator $T$:
\begin{equation*}
\sigma(T) = \sigma_{disc}(T)\cup \sigma_{ess}(T),
\end{equation*}
where we define $\sigma_{disc}$ and $\sigma_{ess}$ below.
\begin{definition}
        For any bounded linear operator $T$ on a separable Hilbert space $\mathcal{H}$ we define the discrete spectrum of T, $\sigma_{disc}(T)$, to be the set of $\lambda \in \sigma(T)$ such that following conditions holds:
        \begin{enumerate}
        \item[$\bullet$] $\lambda$ is isolated in $\sigma(T)$,
        \item[$\bullet$] The Riesz projector $P= \frac{1}{2\pi i} \oint_{\gamma} \frac{dz}{z-T}$, where $\gamma$ is a small circle around $\lambda$, has finite rank,
        \item[$\bullet$] $\lambda-T$ is invertible on the invariant subspace Ker$P$ = Im(I-P),
        \end{enumerate}
        The essential spectrum of $T$ is defined by $\sigma_{ess}(T):= \sigma(T)\setminus \sigma_{disc}(T).$
        \end{definition}

We denote the essential spectral radius of an bounded linear operator $T$ by $r_{ess}(T):= \sup\{|\lambda|: \lambda \in \sigma_{ess}(T)\}$.

Let $X$ be a separable Hilbert space.  We define an appropriate norm on the quotient space $\mathcal{L}(X) / \mathfrak{S_{\infty}}$ where $\mathfrak{S_{\infty}}$ is the ideal of compact operators.
\begin{definition}
For any $T \in \mathcal{L}(X)$
\begin{equation}
\|T\|_{\mathcal{K}}= \inf_{K\in \mathfrak{S_{\infty}}} \|T + K\|_{\mathcal{L}(X)}.
\end{equation}
\end{definition}

\noindent The seminorm $\|\cdot \|_{\mathcal{K}}$ on $\mathcal{L}(X)$ is the canonic norm on the quotient space $\mathcal{L}(X) / \mathfrak{S_{\infty}}$.  We can compute the essential spectral radius of a bounded operator with this norm:
\begin{theorem}[Nussbaum]
\label{Nussbaum}
For any $T \in \mathcal{L}(X)$
\begin{equation}
r_{ess}(T) = \lim_{n \to \infty}(\|T^n\|_{\mathcal{K}})^{\frac{1}{n}}.
\end{equation}
\end{theorem}  
Since $G(t)\circ \opeps\Bigl[\chi\big(\frac{\xi}{\sqrt{\varepsilon}}\bigr)\Bigr]$ is a compact operator, we have $r_{ess}(G(t)) = r_{ess}\Bigl(G(t)\circ \opeps\Bigl[1-\chi\big(\frac{\xi}{\sqrt{\varepsilon}}\bigr)\Bigr]\Bigr)$.  Thus, to determine the essential spectral radius, it suffices to consider linear evolution on high frequencies.

To approximate the linear evolution operator acting on high frequency vector fields, $G(t)\circ \opeps\Bigl[1-\chi\big(\frac{\xi}{\sqrt{\varepsilon}}\bigr)\Bigr]$, we first introduce the parallel transport operator.  Let $g^t:\T^n\to \T^n$ denote the flow map defined by trajectories of the following ODE:
\begin{equation*}
\frac{d}{dt}g^tx = u(g^tx), \hspace{.5 cm} g^0 = \text{Id}.
\end{equation*}
Define $\mathfrak{g}_u(t)$ to be the evolution operator for the equation
\begin{equation}\label{paralleltrans}\left\{\begin{array}{l}
\dot{Y} = -u\cdot \nabla Y,\\
Y(x,0) = Y_0(x)\in L^2(\T^n).\end{array}\right.
\end{equation}
Solutions to (\ref{paralleltrans}) are parallel transport of the initial data $Y_0$ along the flow trajectories:  $\mathfrak{g}_u(t)Y_0(x) = Y_0(g^{-t}x)$.

We must also introduce the matrix-valued function $a_0$ to define the symbol of a pseudodifferential operator that, when composed with parallel transport along the flow, approximates $G(t)$ on high frequencies.  Let $a_0(x, \xi, t)\in M_{n\times n}$, for $(x,\xi,t) \in \T^n \times \R^n \backslash \{0\} \times \R$ and $n = 2,3$, be a solution to
\begin{equation}\renewcommand{\arraystretch}{1.50}\label{a0}\left\{\begin{array}{l}
\dot{a_0} = -\nabla_u a_0 -\pupx a_0 + 2 \frac{\xi \otimes \xi}{|\xi|^2}\bigl( \pupx a_0 \bigr), \\
a_0(x, \xi, 0) = \bigl( 1 - \frac{\xi \otimes \xi}{|\xi|^2}\bigr)\cdot\bigl( 1 - \chi \bigl( \frac{\xi}{\sqrt{\varepsilon}} \bigr) \bigr), \end{array}\right.
\end{equation}
where $\nabla_u$ is the Lie derivative computed in the cotangent bundle $T^*(\T^n)$ along flow trajectories:
\begin{equation*}
\nabla_u := \frac{d}{dt}|_{t=0}(g^t, (g^{-t}_*)^*).
\end{equation*}
In coordinates $\nabla_u = (u, -\pupx^T\xi)$.  Let $G_\varepsilon ^s (t):L^2_{sol} \to L^2_{sol}$ be defined by 
\begin{equation*}
G_\varepsilon ^s (t)w_0 =  \opeps ^s [a_0] \circ \mathfrak{g}_u(t)w_0,
\end{equation*}
where in $\R^3$
\begin{equation*}
(\opeps^s[a_0]w)(x) = \nabla\times\frac{\varepsilon}{(2\pi\varepsilon)^3}\int \frac{i\xi}{|\xi|^2}\times a_0(x,\xi,t) e^{i\xi\frac{x-y}{\varepsilon}}w(y)d^3yd^3\xi.
\end{equation*}

\noindent In \cite{vishik1}, Vishik proves that $G_\varepsilon ^s (t)$ approximates $G(t)$ on high frequencies in the following sense:
\begin{theorem}
\label{VishikThm}
Let $G(t)$ be the evolution operator associated with Euler's equation linearized at $u$.  Then for all $t \geq 0, G_\varepsilon^s(t)$ is a bounded operator in $L^2_{sol}$ and for any fixed $T>0$
\begin{equation}
\| G(t)\circ \opeps\Bigl[1-\chi\big(\frac{\xi}{\sqrt{\varepsilon}}\bigr)\Bigr] -  G_\varepsilon^s(t)\|_{\mathcal{L}(L^2_{sol}, L^2)} = O(\sqrt{\varepsilon}), \hspace{.2 in}0\leq t \leq T,
\end{equation}
with the constant in $O$ uniform over the interval $[0,T]$.
\end{theorem}

To see the connection between solutions $a_0$ to (\ref{a0}) and solutions to (BAS), we introduce a decomposition of our symbol $a_0$:
\begin{equation}
a_0(x,\xi,t) = A_0(x,\xi,t)\bigl(1 -X \bigl(x,\xiover,t \bigr) \bigr),
\end{equation}
where $A_0$ is a solution to the following system:
\begin{equation}
\label{A0}
\renewcommand{\arraystretch}{1.75}
\left\{\begin{array}{l}\dot{A_0} = -\nabla_u A_0 -\pupx A_0 + 2\frac{\xi \otimes \xi}{|\xi|^2} \pupx A_0,\\
A_0(x,\xi,0) = 1 - \frac{\xi \otimes \xi}{|\xi|^2}.
\end{array}\right.
\end{equation}
And $X$ satisfies
\begin{equation}
 \label{X}
\renewcommand{\arraystretch}{1.75}\left\{\begin{array}{l} \dot{X} = -\nabla_u X,\\ X(x,\xi,0)= \chi (\xi).\end{array}\right.
\end{equation}
The matrix symbol $a_0(x,\xi,t)$ forms a strongly continuous cocycle over the flow $(g^{t}\cdot,(g^{-t}_*(x))^*\cdot)$ on the cotangent bundle $T^*(\T^n)$.  Similarly, $A_0(x,\xi,t)$ forms a strongly continuous cocyle on $\T^n \times \R P^{n-1}$.  An important consequence of this fact is that the Lyapunove-type exponent in Theorem \ref{vishik} is well defined.  Solutions to (BAS) are solutions to (\ref{A0}) for $A_0(\cdot,\cdot,t)$ along characteristics which are the flow lines $(g^{t}\cdot,(g^{-t}_*(x))^*\cdot)$ in $\T^n \times \R P^{n-1}$.  Thus it follows that for any initial conditions for (BAS) $(x_0,\xi_0,b_0) \in T^*(\T^n) \times \R^n$, the corresponding solution $b(x_0,\xi_0,b_0;t)$ satisfies
\begin{equation}
\label{Atob}
b(x_0,\xi_0,b_0;t) = A_0(g^{t}x_0, (g^{-t}_*(x))^*\xi_0, t)b_0.
\end{equation}

We finish the discussion of (BAS) by stating a simple property of solutions that will be necessary for proving the main theorem of this paper.  Let $\xi(t)$ satisfy the $\xi$-equation in (BAS) at a time $t>0$ and let $b(t)$ be a solution with the same initial conditions $(x_0, \xi_0)$.  Then we may compute
\begin{equation*}
 \frac{d}{dt}\bigl(b(t), \xi(t)\bigr) =\ \bigl(\dot{b}(t), \xi(t)\bigr) + \bigl(b(t),\dot{\xi}(t)\bigr)\\
=\ \bigl(\pupx b(t), \xi(t)\bigr) + \bigl(b(t), -\bigl(\pupx\bigr)^T\xi(t)\bigr) = 0.
\end{equation*}
Thus, whenever $b_0\perp \xi_0$ we have for any $t>0$,
\begin{equation}
\label{bperpxi}
 (b(t), \xi(t)) = (b(x_0,\xi_0,b_0;t),(g^{-t}_*(x_0))^*\xi_0 ) = 0.
\end{equation}

Next we provide a definition of our $\varepsilon$-pseudodifferential operators and prove a technical lemma that will be necessary for the main results of this paper.
\begin{definition}
For $\T^n = \R^n/2\pi\Z^n$ the class of symbols $S^m_{\rho,\delta}(\T^n)$ denotes the space of functions $\sigma \in C^{\infty}(\T^n\times\R^n)$ such that for all $\alpha$, $\beta \in \Z^n$ there is a constant $C_{\alpha,\beta}$ such that for any $(x,\xi)\in \T^n\times \R^n $
\begin{equation*}
|\partial^{\alpha}_{x}\partial^{\beta}_{\xi}\sigma(x,\xi)| \leq C_{\alpha,\beta}(1+|\xi|)^{m-\rho|\beta| + \delta|\alpha|}.
\end{equation*}
\end{definition}

It follows directly from the definition above that if $\sigma \in C^{\infty}(\T^n\times\R^n)$ is positively homogeneous of degree $m$ in the region $|\xi| \geq R$ for some $R>0$ (that is, $\sigma(x, \lambda \xi) = \lambda^m \sigma(x,\xi), \lambda \geq 1, |\xi|\geq R$), then $\sigma \in S^m_{1,0}(\T^n)$.

For any $\varepsilon > 0$ and $\sigma\in S^0_{\rho,\delta}(\T^n)$ where $0\leq \delta < \rho \leq 1$ define $\text{op}_{\varepsilon}[\sigma(x,\xi)]:\mathcal{D}(\T^n)\to \mathcal{D}(\T^n)$ by equation (\ref{opepsdef}).  If $\sigma \in S^m_{1,0}(\T^n)$ for $m \leq 0$, the psuedodifferential operator $\opeps[\sigma(x,\xi)]$ is a bounded linear operator on $L^2_{sol}(\T^n)$.  A proof for periodic operators is given in \cite{s&v} for example.

We will need the following variant of the Calderon and Vaillancourt theorem from \cite{cv} for $x$-periodic amplitudes to estimate the norms of some $\varepsilon$-pseudodifferential operators (see also \cite{bdem}).

\begin{theorem}[Calderon-Vaillancourt]
Let $\sigma(x,\xi) \in C^{\infty}(\T^n \times \R^n)$ for $0\leq \rho < 1$, satisfy the following inequalities.
\begin{equation*}
\bigl|\partial_{x}^{\alpha}\partial_{\xi}^{\beta} \sigma (x,\xi)\bigr| \leq C_{\alpha \beta}(1 + |\xi|)^{\rho(|\alpha|-|\beta|)},
\end{equation*}
for all $(x,\xi)\in \T^n \times \R^n$, and $(\alpha,\beta) \in \Z^n$.  Then the pseudodifferential operator $\text{op}_1[\sigma(x,\xi)]$ extends from Schwartz space $\mathcal{D}(\T^n) = C^{\infty}(\T^n)$ to $L^2(\T^n)$ and defines a bounded operator there, moreover:
\begin{equation*}
\|\text{op}_1[\sigma]\|_{\mathcal{L}(L^2)} \leq C(n) \sum_{\substack{|\alpha| \leq 2((n/2)+1)\\
|\beta| \leq 2((n/(1-\rho))+1)}}C_{\alpha \beta}.
\end{equation*}
\end{theorem}
\begin{lemma}
\label{sigmaS-1}
Let $\sigma_{\varepsilon}(x,\xi) \in S^{-m}_{1,0}(\T^n)$ for $m > 0$.  Suppose that for some positive constant $c_0$, $\sigma_{\varepsilon}(x,\xi) = 0$ whenever $|\xi| < \frac{c_0}{\sqrt{\varepsilon}}$.  Then $\|\text{op}_1[\sigma_{\varepsilon}]\|_{\mathcal{L}(L^2)} = O(\sqrt{\varepsilon}^m)$.
\end{lemma}

\begin{proof} We will use the Calderon-Vaillancourt inequality to estimate the $L^2$-operator norm of $\text{op}_1[\sigma_{\varepsilon}(x,\xi)]$.  Let $\beta$, $\gamma \in \Z^n$.  Since $\sigma_{\varepsilon}(x,\xi)\in S^{-m}_{1,0}$, there is some constant $C_{\beta,\gamma}$ such that for any $x \in \T^n$
\begin{equation*}
 |\partial_x^{\beta}\partial_{\xi}^{\gamma}\sigma_{\varepsilon}(x,\xi)| \leq C_{\beta,\gamma}(1+|\xi|)^{-m -|\gamma|}.
\end{equation*}
Multiply this inequality by $(1+|\xi|)^{1/2(|\gamma|-|\beta|)}$ to get
\begin{align*}
|\partial_x^{\beta}\partial_{\xi}^{\gamma}\sigma_{\varepsilon}(x,\xi)|
(1+|\xi|)^{1/2(|\gamma|-|\beta|)}
\leq& C_{\beta,\gamma}(1 + |\xi|)^{-(m+1/2(|\beta| + |\gamma|))}\\
\leq& C_{\beta,\gamma}(\frac{\sqrt{\varepsilon}}{c_0})^{m+1/2(|\beta| + |\gamma|)}.
\end{align*}
This last inequality follows from the fact that the symbol $\sigma_{\varepsilon}(x,\xi)=0$ for $|\xi|<\frac{c_0}{\sqrt{\varepsilon}}$.  So for any $(x,\xi) \in \T^n \times \R^n$ we have
\begin{equation*}
 |\partial_x^{\beta}\partial_{\xi}^{\gamma}\sigma_{\varepsilon}(x,\xi)| \leq C_{\beta,\gamma}(\frac{\sqrt{\varepsilon}}{c_0})^{m+1/2(|\beta| + |\gamma|)}(1 + |\xi|)^{1/2(|\beta|-|\gamma|)}.
\end{equation*}
Thus, we may use the Calderon-Vaillancourt inequality for $\rho = 1/2$ to estimate the norm of our operator.  The most substantial contribution to the norm is the $\beta=\gamma=0$ summand.  Therefore we have
\begin{equation*}
\|\text{op}_1[\sigma_{\varepsilon}(x,\xi)]\|_{\mathcal{L}(L^2_{sol})}= O(\sqrt{\varepsilon}^m).
\end{equation*}
\end{proof}

In the proof of the main theorem we deal with a rougher estimate of the linear evolution operator in terms of an $\varepsilon$-pseudodifferential operator.  We define $G_{\varepsilon}(t)\psi_{\delta}(x) := (\opeps[a_0]\circ\mathfrak{g}_u(t) \psi_{\delta})(x)$.  The advantage of looking at vector fields such as $\psi_{\delta}$ defined in Section \ref{3Dclassifysection}, is that we can estimate $G_{\varepsilon}(t)\psi_{\delta}$ explicitly, which we will see in this next lemma.   We omit the proof, which can be found in \cite{vishik1}.  
\begin{lemma}
\label{imagepsilemma}
 Let $\psi_{\delta}$ be defined as in line (\ref{definepsi}) above.  Then for any fixed $t>0$, we have the following approximation for $G_{\varepsilon}(t)\psi_{\delta}(x) := (\opeps[a_0]\circ\mathfrak{g}_u(t) \psi_{\delta})(x)$:
\begin{equation*}
 (\opeps[a_0]\circ\mathfrak{g}_u(t) \psi_{\delta})(x) = h_0(g^{-t}x)A_0(x, (g^{-t}_*(x))^*\xi_0,t)Pe^{ig^{-t}x\cdot\xi_0/\delta} + r_{\delta}(x),
\end{equation*}
where $A_0$ is the homogeneous part of $a_0$ defined by (\ref{A0}) and $\|r_{\delta}\|_{L^2} = O(\delta)$.
\end{lemma}

\begin{remark}
\label{Atobremark}
From equation (\ref{Atob}) we have
\begin{equation*}
h_0(g^{-t}x)A_0(x, (g^{-t}_*(x))^*\xi_0,t)Pe^{ig^{-t}x\cdot\xi_0/\delta} = h_0(g^{-t}x)b(g^{-t}x,\xi_0,P;t)e^{ig^{-t}x\cdot\xi_0/\delta}.
\end{equation*}
\end{remark}

Now we present, without proof,  a slightly generalized 2-dimensional version of Lemma \ref{imagepsilemma} to approximate the linear evolution of our $\phi_{\zeta,\delta}$ vector fields from Section \ref{2Dclassifysection}, where it is no longer assumed that the frequency vector $\xi_0$ has integer components.  The parameter $\zeta$ is introduced to ensure that our estimate is periodic (see Remark \ref{zetarmk}).  The proof is completely similar to the proof of Lemma \ref{imagepsilemma} found in \cite{vishik1}.

\begin{lemma}
\label{imagephilemma}
 Let $h_0\in C^{\infty}(\T^2)$ be supported on $B_1(0)$, the ball centered at $0$ of radius $1$.  For $0<\zeta<1$ and fixed $x_0\in \T^2$, define $h_{\zeta}$ by
\begin{equation*}
 h_{\zeta}(x) := h_0\bigl(\frac{x-x_0}{\zeta}\bigr).
\end{equation*}
Let $\xi_0 \in \R^2$, $\delta^{-1}\in \Z_+$ and define $\phi_{\zeta,\delta}(x):= -i\delta\nabla^{\perp}(h_{\zeta}e^{ix\cdot \xi_0/\delta})$.  Then for any fixed $t>0$ we can approximate $G_{\varepsilon}(t)\phi_{\zeta,\delta}(x) := (\opeps[a_0]\circ\mathfrak{g}_u(t) \phi_{\zeta,\delta})(x)$ as follows:
\begin{equation*}
 (\opeps[a_0]\circ\mathfrak{g}_u(t) \phi_{\zeta,\delta})(x) = h_{\zeta}(g^{-t}x)b(g^{-t}x,\xi_0,\xi_0^{\perp};t)e^{ig^{-t}x\cdot\xi_0/\delta} + r_{\delta}(x) ,
\end{equation*}
where $\|r_{\delta}\|_{L^2} = O(\delta)$.
\end{lemma}

\section{Main theorems for 3-dimensional flows}
\label{mainthms3D}
In this section we prove the main theorem for 3-dimensional flows, Theorem \ref{3Dmainthm}.
In the following theorem $b(x_0,\xi_0,b_0;t)$ is a solution to (BAS) corresponding to our equilibrium flow $u$, with initial conditions $(x_0,\xi_0,b_0)$.  Recall the set of admissible initial conditions is
\begin{equation*}
\mathcal{A}:= \{(x_0,\xi_0,b_0)\in \T^3 \times \R^3 \times \R^3 |\ \xi_0 \perp b_0 ,\ |\xi_0|=|b_0|=1 \}.
\end{equation*}
We also denote the vorticity vector field $\omega = \curl(u)$.
\begin{theorem}
\label{3Dmainthm}
\begin{itemize}
 \item[(i)]   Let $\mu_{3*} \in\R$ be defined by
\begin{equation*}
\mu_{3*} = \lim_{t\to\infty}\frac{1}{t}\log\sup_{\substack {(x_0.\xi_0,b_0)\in \mathcal{A}\\x_0 \in \supp(\omega)}} |b(x_0, \xi_0, b_0; t)|,
\end{equation*}Then $e^{\mu_{3*}t} \leq r_{ess}(G(t)|_{\overline{\text{Im}B}})$.
\item[(ii)] If $\supp(\omega)$ is a proper subset of the fluid domain $\T^3$, let $\mu_{3F} \in\R$ be defined by
\begin{equation*}
\mu_{3F} = \lim_{t\to\infty}\frac{1}{t}\log\sup_{\substack {(x_0,\xi_0,b_0)\in \mathcal{A}\\x_0 \notin \supp(\omega)}} |b(x_0, \xi_0, b_0; t)|,
\end{equation*}
Then $e^{\mu_{3F}t} \leq r_{ess}(G_F(t))$, where $G_F(t)$ denotes $G(t)$ on the factor space.
\end{itemize}
\end{theorem}

We need to work with the following seminorm to prove inequalities involving $\|\cdot\|_{\mathcal{L}(F)}$:

\begin{definition}
\label{F-norm}Let $\mathbb{P}: L^2_{sol} \to \text{Ker}B$ denote the orthogonal projection onto $\text{Ker}B$.  We define the $\mathcal{F}-$seminorm, $\|\cdot\|_{\mathcal{F}}$, on $\mathcal{L}(L^2_{sol})$ by $\|S\|_{ \mathcal{F}} := \|\mathbb{P}S\mathbb{P}\|_ {\mathcal{L}(L^2_{sol})}$.
\end{definition}

\begin{remark}\label{F-normrmk}  Because $L^2_{sol} = \overline{\text{Im}B} \oplus \text{Ker}B$, if $T \in \mathcal{L}(L^2_{sol})$ leaves $\overline{\text{Im}B}$ invariant, we have
\begin{equation}
 \label{FnormT}
\|T\|_{\mathcal{L}(F)}= \sup_{\substack{x\in L^2_{sol}\\ \mathbb{P} x\neq 0}}\frac{\|\mathbb{P} T \mathbb{P} x\|_{L^2_{sol}}}{\|\mathbb{P} x\|_{L^2_{sol}}} =  \| T\|_{\mathcal{F}}.
\end{equation}    
\end{remark}

Before proving Theorem \ref{3Dmainthm} we prove a proposition:
\begin{prop}
\label{lowerbnd}  
Fix $T>0$.
\begin{itemize}
\item[(i)] Let $\Theta_*(t)$ denote the following quantitity:
\begin{equation*}
\Theta_*(t) =  \sup_{\substack {(x_0,\xi_0,b_0)\in \mathcal{A}\\x_0 \in \supp(\omega)}} |b(x_0, \xi_0, b_0; t)|.
\end{equation*}
Then for any $\varepsilon > 0$ and $t\in[0,T]$
\begin{equation*}
\|G_{\varepsilon}^s(t)\|_{\mathcal{L}(\overline{\text{Im}B},L^2)} + O(\sqrt{\varepsilon})\geq \Theta_*(t),
\end{equation*}
where the constant in O is uniform for $t\in [0,T]$.
\item[(ii)] Whenever $\supp(\omega)$ is a proper subset of the fluid domain, $\T^3$, define $\Theta_F(t)$ by
\begin{equation*}
\Theta_F(t) =  \sup_{\substack {(x_0,\xi_0,b_0)\in \mathcal{A} \\x_0 \notin \supp(\omega)}} |b(x_0, \xi_0, b_0; t)|,
\end{equation*}
Then for any $\varepsilon > 0$ and $t\in[0,T]$
\begin{equation*}
\|G_{\varepsilon}^s(t)\|_{\mathcal{F}} + O(\sqrt{\varepsilon})\geq \Theta_F(t).
\end{equation*}
where the constant in O is uniform for $t\in [0,T]$.
\end{itemize}
\end{prop}

\begin{remark}
\label{limitexist}
Because $\{A_0(x,\xi,t):(x,\xi)\in T^*(\T^n), t\geq0\}$ is a strongly continuous cocyle over the flow $\{g^t\}_{t\in\R}$, we have that $\log\Theta_*(t)$and $\log \Theta_F(t)$ are subadditive, which implies that both limits from the statement of Theorem \ref{3Dmainthm} exist.
\end{remark}
To prove this proposition we will choose appropriate sequences of fast oscillating vector fields (one that is almost in $\overline{\text{Im}B}$ and one that is in $\text{Ker}B$) and show that the appropriate norms of their images under $G_{\varepsilon}^s(t)$ approach $\Theta_*(t)$ and $\Theta_F(t)$, respectfully, from below.

\begin{proof}[Proof of Proposition \ref{lowerbnd}] First we prove part (\textit{i}).  Choose $x_0 \in \T^3$ and $\xi_0 \in \Z^3$ such that $(\omega(x_0),\xi_0) \neq 0$ and $h_0 \in C^{\infty}(\T^3)$ with $\supp h_0 \subset B_0(1)$ and $h_0(0) = 1$. Let $0 < \zeta < 1$ and define $h_{\zeta} \in C^{\infty}(\T^3)$ by
\begin{equation*}
 h_{\zeta} := h_0\Bigl(\frac{x-x_0}{\zeta}\Bigr).
\end{equation*}
 Then by Lemma \ref{inimage} there exists $\overline{\psi}_{\zeta,\delta} \in L^2_{sol}$ such that
\begin{equation}
\label{Bpsizd}
B(\overline{\psi}_{\zeta,\delta})(x) = \psi_{\zeta,\delta}(s) + r_{\zeta} + r_{\delta},
\end{equation}
where the $\|r_{\zeta}\|_{L^2}\leq c_0\zeta^{5/2}$ for $c_0$ independent of $\delta$ and $\|r_{\delta}\|_{L^2} = O(\delta)$.
Then if we expand $\psi_{\zeta,\delta}$ as in line (\ref{psiexpanded}), we have
\begin{equation}
\label{Bpsi}
B(\overline{\psi}_{\zeta,\delta})(x) = h_{\zeta}(x)Pe^{ix\cdot\xi_0/\delta} + r_{\zeta} + \overline{r}_{\delta},
\end{equation}
where
\begin{equation*}
 \overline{r}_{\delta} = r_{\delta} + \delta\bigr[\nabla h_{\zeta}(x) \times \Bigr( \frac{i\xi_0\times P}{|\xi_0|^2} \Bigl)e^{ix\cdot\xi_0/\delta}\bigl].
\end{equation*}
It follows that $\|\overline{r}_{\delta}\|_{L^2} = O(\delta)$.
Apply Lemma \ref{imagepsilemma} to the main order term in the expansion (\ref{Bpsizd}) for $B(\overline{\psi}_{\zeta,\delta})$ to estimate
\begin{align*}
 (\opeps[a_0]\circ\mathfrak{g}_u(t)& B(\overline{\psi}_{\zeta,\delta}))(x)\\ &= h_{\zeta}(g^{-t}x)b(g^{-t}x,\xi_0,P;t)e^{ig^{-t}x\cdot\xi_0/\delta} + \tilde{r}_{\zeta}(x) + \tilde{r}_{\delta}(x),
\end{align*}
where $\tilde{r}_{\delta} = \opeps[a_0]\circ\mathfrak{g}_u(t)\overline{r}_{\delta}$ and $\tilde{r}_{\zeta} = \opeps[a_0]\circ\mathfrak{g}_u(t)r_{\zeta}$.  Hence $\|\tilde{r}_{\delta}\|_{L^2} = O(\delta)$ and $\|\tilde{r}_{\zeta}\|_{L^2} \leq \tilde{c}_0\zeta^{5/2}$ where $\tilde{c}_0 := c_0\|\opeps[a_0]\|_{\mathcal{L}(L^2)}$ does not depend on $\delta$.  It follows that
\begin{equation*}
 \lim_{\delta \to 0}\|(\opeps[a_0]\circ\mathfrak{g}_u(t) B(\overline{\psi}_{\zeta,\delta}))(x)\|_{L^2} = \|h_{\zeta}(g^{-t}x)b(g^{-t}x,\xi_0,P;t)\|_{L^2} + O(\zeta^{5/2}).
\end{equation*}
Then from line (\ref{Bpsi}) we have $\|B(\overline{\psi}_{\zeta,\delta})\|_{L^2} = \|h_{\zeta}P\|_{L^2} + O(\zeta^{5/2}) + O(\delta)$, thus we may estimate
\begin{align}
\nonumber
\|\text{op}_\varepsilon[a_0] \circ \mathfrak{g}_u(t)&\|_{\mathcal{L}(\overline{\text{Im}B},L^2)} + O(\zeta^{5/2})\\
 &\geq \sup_{\substack{x_0\in \T^3, \xi_0 \in\Z^3\\ (\omega(x_0),\xi_0)\neq 0\\P\perp\xi_0 }}
\frac{\|h_{\zeta}(g^{-t}x)b(g^{-t}x,\xi_0,P;t)\|_{L^2}}{\|h_{\zeta}P\|_{L^2}}\\
\label{Imsup}
&= \sup_{\substack{x_0\in \supp (\omega), \xi_0\in \Z^3\\P\perp\xi_0 }}
\frac{\|h_{\zeta}(g^{-t}x)b(g^{-t}x,\xi_0,P;t)\|_{L^2}}{\|h_{\zeta}P\|_{L^2}}
\end{align}
where the equality in the second line comes from taking the closure of the pairs $(x_0,\xi_0) \in \T^3\times \Z^3$ such that $(\omega(x_0),\xi_0) \neq 0$.  Next we take the limit as $\zeta \to 0$.  The flow map $g^t$ is measure preserving, so composition with it will not affect the norm in $L^2$.  Also $h_{\zeta}(x_0) = 1$ and $b(\cdot,\cdot,P;t)$ depends linearly on $P$, so if we take the limit in $\zeta$ of the expression in (\ref{Imsup}) we have
\begin{equation*}
 \lim_{\zeta \to 0} \frac{\|h_{\zeta}(g^{-t}x)b(g^{-t}x,\xi_0,P;t)\|_{L^2}}{\|h_{\zeta}P\|_{L^2}} = |b(x_0,\xi_0,\frac {P}{|P|};t)|.
\end{equation*}
We can approximate any $\xi \in \R^3$ by $\xi_0 \in \Z^3$ and $b$ is homogeneous of degree 0 in $\xi_0$, so it is equivalent to take the supremum in the RHS of (\ref{Imsup}) over $\xi_0 \in \R^3$ with $|\xi_0| = 1$.  Hence
\begin{equation}
\label{limitA}
 \|\text{op}_\varepsilon[a_0] \circ \mathfrak{g}_u(t)\|_{\mathcal{L}(\overline{\text{Im}B},L^2)} \geq \sup_{\substack{(x_0, \xi_0, b_0)\in \mathcal{A}\\ x_0 \in \supp(\omega)}} |b(x_0, \xi_0, b_0; t)| = \Theta(t).
\end{equation}
Therefore,
\begin{equation*}
\|\text{op}_\varepsilon[a_0] \circ \mathfrak{g}_u(t)\|_{\mathcal{L}(\overline{\text{Im}B},L^2)} \geq \Theta_*(t).
\end{equation*}

Now we prove an estimate for part (\textit{ii}).  Recall, the factor space $F:= L^2_{sol}/\overline{\text{Im}B}$.  Consider a vector field $\psi_{\delta} \in C^{\infty}_{sol}(\T^3)$, defined as in (\ref{definepsi}) with a condition on its support to ensure it will be in $\text{Ker}B$:
 \begin{equation}
\psi_{\delta}(x) = \delta \nabla \times \left(\frac{i\xi_0 \times P}{|\xi_0|^2}h_0(x)e^{ix\cdot\xi_0/\delta}\right),
\end{equation}
where $\xi_0 \in \Z^3, \delta ^{-1} \in \Z_+, P \perp \xi_0$ is a constant vector and $h_0 \in C^{\infty}(\T^3)$ is an arbitrary smooth scalar function with $\supp (h_0)$  disjoint from $\supp(\omega)$.   This implies that $\supp (\psi_{\delta})$ is disjoint from $\supp(\omega)$.  It follows that $\psi_{\delta}\in \text{Ker}B$.  If we apply Lemma \ref{imagepsilemma} to $\psi_{\delta}$ we have
\begin{equation}
 \label{super3}
\|\text{op}_{\varepsilon}[a_0]\circ\mathfrak{g}_u(t) \psi_{\delta}\|_F = \|h_0(g^{-t}\cdot)b(g^{-t}\cdot,\xi_0,P;t)e^{i(\cdot)\cdot\xi_t/\delta}\|_{F} + O(\delta),
\end{equation}
where $\|\cdot\|_F$ denotes the canonical factor space norm.
The complement of $\supp(\omega)$ is invariant under the flow $g^t$, so we have $\supp (h \circ g^{-t})$ is also disjoint from $\supp (\omega)$.  Hence
\begin{equation*}
h_0(g^{-t}x)b(g^{-t}x,\xi_0,P;t)e^{ig^{-t}x\cdot\xi_0/\delta} \in \text{Ker}B.
\end{equation*}
It follows that
\begin{equation*}
\frac{\|\mathbb{P}(h_0(g^{-t}x)b(g^{-t}x,\xi_0,P;t)e^{ig^{-t}x\cdot\xi_0/\delta})\|_{F}}{\|\mathbb{P}(\psi_{\delta})\|_{F}} = \frac{\|h_0(g^{-t}x)b(g^{-t}x,\xi_0,P;t)e^{ig^{-t}x\cdot\xi_0/\delta}\|_{L^2}}{\|\psi_{\delta}\|_{L^2}}.
\end{equation*}
Now consider equation (\ref{super3}) and take the limit as $\delta \to 0$ and we have
\begin{equation}
\label{Kersup}
\|\text{op}_\varepsilon[a_0] \circ \mathfrak{g}_u(t)\|_{\mathcal{F}}
\geq \sup_{\substack{h_0\in C^\infty(\T^3), \xi_0\in \Z^3 \setminus \{0\}\\
   \text{supp }h_0 \cap \supp(\omega) = \emptyset \\ P\perp \xi_0}}
\frac{\|h_0(g^{-t}x)b(g^{-t}x,\xi_0,P;t)\|_{L^2}}{\|h_0P\|_{L^2}}.
\end{equation}

  We are taking a supremum over all $h_0 \in C^{\infty}(\T^3)$ with $\supp (h_0)$ disjoint from $\supp (\omega)$ and $\T^3 \setminus\supp (\omega)$ is invariant under the flow map, so we can restrict our consideration to $x_0 \notin \supp (\omega)$.  The flow map $g^{-t}$ is measure preserving, so that change of coordinates will not affect the $L^2$-norm.  Also, since $b$ is homogeneous of degree 0 in $\xi_0$ and linear in $P$ we have
\begin{equation*}
 \|\text{op}_\varepsilon[a_0] \circ \mathfrak{g}_u(t)\|_{\mathcal{F}} \geq \sup_{\substack {(x_0,\xi_0,b_0)\in \mathcal{A} \\x_0 \notin \supp(\omega)}} |b(x_0, \xi_0, b_0; t)| = \Theta_F(t).
\end{equation*}

To finish the proof for both classes of perturbations, we must estimate the difference:
\begin{equation*}
\|G_\varepsilon^s(t) - \opeps[a_0] \circ \mathfrak{g}_u(t)\|_{\mathcal{L}(L^2)}.
\end{equation*}
We may simply extend the definition of $G_{\varepsilon}^s$ to all of $L^2$ so that for $v\in L^2$
\begin{equation*}
 (G^s_{\varepsilon}(t)v)(x) = \nabla_x \times \frac{\varepsilon}{(2\pi\varepsilon)^3}
\int \frac{i\xi}{|\xi|^2}\times a_0(x,\xi,t)\mathfrak{g}_u(t)v(y)e^{i(x-y)\cdot\xi/\varepsilon}dyd\xi.
\end{equation*}
Notice that the matrix $a_0(x,\xi,t)$ maps into $\xi^{\perp}$ for all $t$.  Since $i\xi \times (i\xi \times w)= w$ whenever $w \perp \xi$, this implies
\begin{equation*}
 \opeps[a_0]\circ \mathfrak{g}_u(t)v = \nabla_x(e^{ix \cdot \xi/\varepsilon}) \times \frac{\varepsilon}{(2\pi\varepsilon)^3}\int \frac{i\xi}{|\xi|^2}\times a_0(x,\xi,t)\mathfrak{g}_u(t)v(y)e^{-iy\cdot\xi/\varepsilon}dyd\xi .
\end{equation*}
Hence
\begin{align*}
 G_\varepsilon^s(t) - \opeps[a_0] \circ \mathfrak{g}_u(t)=& \varepsilon\opeps \Bigl[\nabla_x \times \Bigl(\frac{i\xi}{|\xi|^2}\times a_0  \Bigr)\Bigr]\circ \mathfrak{g}_u(t)\\
\label{opD}
=& \text{op}_1\Bigl[\nabla_x \times \Bigl(\frac{i\xi}{|\xi|^2}\times a_0(x,\varepsilon\xi,t)\Bigr)\Bigr]\circ \mathfrak{g}_u(t).
\end{align*}
Consider the symbol $D(x,\xi,t)$ defined by
\begin{equation*}
D(x,\xi,t):= \nabla_x \times\Bigl(\frac{i\xi}{|\xi|^2}\times a_0(x,\varepsilon\xi,t)  \Bigr).
\end{equation*}
For large $|\xi|$, $D(x,\xi,t)$ has homogeneity of order $-1$.  Since $a_0(x,\varepsilon\xi,t) = \bigl(1 - X(x,\sqrt{\varepsilon}\xi,t)\bigr)A_0(x,\xi,t)$, it follows that there is some constant $c(T)$ that depends on $T$ only such that for any $t\in[0,T]$, $D(x,\xi,t) = 0$ whenever $|\xi| < \frac{c(T)}{\sqrt{\varepsilon}}$.  We also note that for any $\beta,\ \gamma \in \Z^3$, there exists a constant $C_{\beta,\gamma}(T)$ such that
\begin{equation*}
 |\partial_x^{\beta} \partial_\xi^{\gamma}D(x,\xi,t)| \leq C_{\beta,\gamma}(T)(1 + |\xi|)^{-1-|\gamma|}\hspace{.3 cm}\text{ for any }t\in[0,T].
\end{equation*}
Now we may apply Lemma \ref{sigmaS-1} to get
\begin{equation*}
\|G_\varepsilon^s(t) - \opeps[a_0] \circ \mathfrak{g}_u(t)\|_{\mathcal{L}(L^2)} =  \|\text{op}_1[D(x,\xi,t)]\|_{\mathcal{L}(L^2)}  = O(\sqrt{\varepsilon}).
\end{equation*}
We remark that in the proof of Lemma \ref{sigmaS-1} the constant in $O$ depends only on the constants $C_{\beta,\gamma}(T)$ and $c(T)$, so $O(\sqrt{\varepsilon})$ is uniform for $t\in[0,T]$.  Then from the definition of the $\mathcal{F}-$seminorm, we have
\begin{equation}
\label{Fnormestimate}
\|G_\varepsilon^s(t) - \opeps[a_0] \circ \mathfrak{g}_u(t)\|_{\mathcal{F}} \leq \|G_\varepsilon^s(t) - \opeps[a_0] \circ \mathfrak{g}_u(t)\|_{\mathcal{L}(L^2_{sol})} = O(\sqrt{\varepsilon}),
\end{equation}
where the constants in $O$ are uniform for $t\in [0,T]$.  This completes the proof.

\end{proof}

Now we prove the main theorem of this Chapter:
\begin{proof}[Proof of Theorem \ref{3Dmainthm}]
We begin with statement \textit{(i)}.  Let $C \in \mathcal{L}(L^2)$ be an arbitrary operator of finite rank.  Then we get the following inequality for any $\varepsilon >0$.
\begin{equation}
\label{GplusC}
\|G(t)+C\|_{\mathcal{L}(L^2)} \geq \|(G(t)+C)\circ \text{op}_{\varepsilon}\bigl[1-\chi\bigl(\frac{\xi}{\sqrt{\varepsilon}}\bigr)\bigr] \|_{\mathcal{L}(L^2)}.
\end{equation}
Since $C$ has finite rank, we may write
\begin{equation*}
C = \sum_{j=1}^M g_j(f_j,\cdot),
\end{equation*}
for some $\{g_j\}_{j=1}^M, \{f_j\}_{j=1}^M \subset L^2$.  Since $\opeps\bigl[1-\chi\bigl(\frac{\xi}{\sqrt{\varepsilon}}\bigr)\bigr]$ is self-adjoint, it follows that
\begin{align*}
\|C \circ \opeps\bigl[1-\chi\bigl(\frac{\xi}{\sqrt{\varepsilon}}\bigr)\bigr]\|_{\mathcal{L}(L^2)} =& \|\sum_{j=1}^M g_j(\opeps\bigl[1-\chi\bigl(\frac{\xi}{\sqrt{\varepsilon}}\bigr)\bigr]f_j,\cdot)\|_{\mathcal{L}(L^2)}\\
=& o(1) \text{ as }\varepsilon \to 0,
\end{align*}
since for each $j = 1,2...M$,
\begin{equation*}
\|\opeps\bigl[1-\chi\bigl(\frac{\xi}{\sqrt{\varepsilon}}\bigr)\bigr]f_j\|_{\mathcal{L}(L^2)} = o(1) \text{  as } \varepsilon \to 0.
\end{equation*}
This implies
\begin{equation}
\label{blah}
 \|C \circ \opeps\bigl[1-\chi\bigl(\frac{\xi}{\sqrt{\varepsilon}}\bigr)\bigr]\|_{\mathcal{L}(\overline{\text{Im}B}, L^2)} = o(1) \text{ as }\varepsilon \to 0.
\end{equation}
Let $N \in \N$ and substitute $Nt$ with $t$ in (\ref{GplusC}).  Then by equation (\ref{blah}) above
\begin{equation*}
\|G(Nt)+C\|_{\mathcal{L}(\overline{\text{Im}B},L^2)} \geq \|G(Nt)\circ \text{op}_{\varepsilon}\bigl[1-\chi\bigl(\frac{\xi}{\sqrt{\varepsilon}}\bigr)\bigr]\|_{\mathcal{L}(\overline{\text{Im}B},L^2)} - o(1)\text{ as }\varepsilon \to 0.
\end{equation*}
From Theorem~\ref{VishikThm} we have
\begin{equation*}
\|G(Nt)+C\|_{\mathcal{L}(\overline{\text{Im}B},L^2)}\geq \|G^s_{\varepsilon}(Nt)\|_{\mathcal{L}(\overline{\text{Im}B},L^2)} -O(\sqrt{\varepsilon})-o(1) \text{ as }\varepsilon\to 0.
\end{equation*}
And Proposition \ref{lowerbnd} implies
\begin{equation*}
\|G(Nt)+C\|_{\mathcal{L}(\overline{\text{Im}B},L^2)}\geq \Theta_*(Nt)-O(\sqrt{\varepsilon})-o(1) \text{ as }\varepsilon\to 0.
\end{equation*}
Letting $\varepsilon \to 0$,
\begin{equation*}
\|G(Nt)+C\|_{\mathcal{L}(\overline{\text{Im}B},L^2)}\geq \Theta_*(Nt).
\end{equation*}
Since $C$ was arbitrary, we have
\begin{equation*}
\|G(Nt)\mid_{\overline{ImB}}\|_{\mathcal{K}} \geq \Theta_*(Nt),
\end{equation*}
where $\|\cdot\|_{\mathcal{K}}$ denotes Nussbaum's seminorm, introduced in Section \ref{approxG(t)}.  Take the $N$th root of both sides of the equation to get
\begin{equation*}
 \|G(Nt)\mid_{\overline{ImB}}\|^{1/N}_{\mathcal{K}}  \geq e^{t\frac{1}{Nt}\log (\Theta_*(Nt))}.
\end{equation*}
If we take the limits as $N\to \infty$, for 3-dimensional flows we have
\begin{equation*}
r_{ess}(G(t)\mid_{\overline{\text{Im}B}}) \geq e^{\mu_{3*} t}.
\end{equation*}
Thus we have the lower bound for $\overline{\text{Im}B}$.

To compute a lower bound for the factor space, we assume $\supp(\omega)$ is a proper subset of the fluid domain, $\T^3$.  In this case we may use Proposition \ref{lowerbnd}.  

For any $x\in L^2_{sol}$, we let $[x]\in F$ denote the equivalence class in $F:= L^2_{sol}/\overline{\text{Im}B}$ represented by $x$.   Any operator $K \in \mathfrak{S}_{\infty}(F)$ can be lifted to an operator $\overline{K}\in \mathfrak{S}_{\infty}$ as follows:  Let $\{\tilde{f}_j\}_{j=1}^{\infty}$ be a Schauder basis for $\text{Ker}B$.  In the canonical sense, $\{[\tilde{f}_j]\}_{j=1}^{\infty}$ is also a Schauder basis for the factor space, $F$. We may write
\begin{equation*}
K  = \sum_{j=1}^{\infty} [\tilde{g}_j]([\tilde{f}_j],\cdot),
\end{equation*}
where $\tilde{g}_j\in \text{Ker}B$ for each $j=1,2...$.  Then we define
\begin{equation}
 \label{Kbar}
\overline{K}  := \sum_{j=1}^{\infty} \tilde{g}_j(\tilde{f}_j,\cdot).
\end{equation}
Notice that $\overline{K}$ leaves $\overline{\text{Im}B}$ invariant and $\overline{K}_F = K$.

Let $\|\cdot\|_{\mathcal{K}(F)}$ be the Nussbaum seminorm on $F$.  Then
\begin{equation*}
 \|T_F\|_{\mathcal{K}(F)} := \inf_{K\in\ \mathfrak{S_{\infty}}(F)}\|T_F + K\|_{\mathcal{L}(F)}.
 \end{equation*}  
We begin with the inequality analogous to  (\ref{GplusC}) for $\|\cdot\|_{\mathcal{L}(F)}$.  For any finite rank operator $C\in \mathcal{L}(F)$ we have
\begin{equation*}
\|G_F(t)+C\|_{\mathcal{L}(F)}\geq\ \|(G(t)+C)\circ \text{op}_{\varepsilon}\bigl[1-\chi\bigl(\frac{\xi}{\sqrt{\varepsilon}}\bigr)\bigr] \|_{\mathcal{F}}.
\end{equation*}  
The argument is completely similar to that for the image, except that for the factor space we must be careful to use the seminorm $\|\cdot\|_{\mathcal{F}}$ whenever we estimate the size of an operator that is not well defined on the factor space.  This leads to 
\begin{equation*}
\|G_F(Nt)\|_{\mathcal{K}(F)} \geq \Theta_F(Nt), 
\end{equation*}
for $N\in\N$.  Take the $N$th root of both sides of the equation, exponentiate the RHS as we did for the image case and then take the limit as $N\to \infty$.  Thus for 3-dimensional flows where  $\supp(\omega)$ is a proper subset of $\T^3$ we have
\begin{equation*}
 r_{ess}(G_F(t)) \geq e^{\mu_{3F}t}.
\end{equation*}
\end{proof}
\begin{remark}
\label{mainthmrmk}
The proof of Theorem \ref{3Dmainthm} did not depend on our flow being 3-dimensional.  In Section \ref{2Dmainthm} we will introduce 2-dimensional propositions similar to Proposition \ref{lowerbnd} and reference the proof of \ref{3Dmainthm} to prove our main theorem for 2-dimensional flows, Theorem \ref{2Dmainthm}.
\end{remark}

We have the following corollaries to Theorem \ref{3Dmainthm}:
\begin{corollary}
\label{firstcorollary}For a 3-dimensional flow with vorticity $\omega$, if $\supp(\omega)$ is a proper subset of $\T^3$, then
 \begin{equation*}
r_{ess}(G(t)) = \max\{r_{ess}(G_F(t)), r_{ess}(G(t)\mid_{\overline{\text{Im}B}})\}.
\end{equation*}
\end{corollary}
\begin{corollary}
\label{secondcorollary}
If the support of $\omega$ is the entire fluid domain, $T^3$, then
\begin{equation*}
r_{ess}(G(t)\mid_{\overline{\text{Im}B}}) = r_{ess}(G(t)).
\end{equation*}
\end{corollary}
Before proving these corollaries, we need the following proposition:
\begin{prop}
\label{ressprop}
For 2- or 3-dimensional flows and for any $t>0$,
\begin{equation*}
 r_{ess}(G|_F(t)) \leq r_{ess}(G(t)).
\end{equation*}
\end{prop}
\begin{proof}

Let $\|\cdot\|_{\mathcal{K}(F)}$ be the Nussbaum seminorm on $F$.  Then from Remark \ref{F-normrmk} we have
\begin{equation*}
 \inf_{K\in\ \mathfrak{S_{\infty}}(F)}\|T_F + K\|_{\mathcal{L}(F)} = \inf_{K\in\ \mathfrak{S_{\infty}}(F)} \|T + \overline{K}\|_{\mathcal{F}},
\end{equation*}
where $\overline{K} \in \mathfrak{S}_{\infty}$ is the lift of $K\in \mathfrak{S}_{\infty}$ defined by (\ref{Kbar}).

Notice that for $C \in \mathfrak{S}_{\infty}$, there is some $K_C\in \mathfrak{S}_{\infty}(F)$ such that $ \overline{K_C} = \mathbb{P} C \mathbb{P}$, where $\overline{K_C}$ denotes the lift of $K_C$ in the sense of (\ref{Kbar}).  Since $\overline{K_C} = \mathbb{P}\overline{K_C}\mathbb{P}$, we have
\begin{equation*}
 \inf_{K\in\ \mathfrak{S_{\infty}}(F)} \|T + \overline{K}\|_{\mathcal{F}} \leq\ \inf_{C \in \mathfrak{S}_{\infty}} \|\mathbb{P}(T + \overline{K_C})\mathbb{P}\|_{\mathcal{L}(L^2_{sol})}
\leq\ \inf_{C \in \mathfrak{S}_{\infty}} \|T + C\|_{\mathcal{L}(L^2_{sol})}.
\end{equation*}
Thus, for any $T \in \mathcal{L}(L^2_{sol})$ which leaves $\overline{\text{Im}B}$ invariant we have $\|T_F\|_{\mathcal{K}(F)} \leq \|T\|_{\mathcal{K}(L^2_{sol})}$.  Thus for any $N\in \N$
\begin{equation*}
 \|G_F(Nt)\|_{\mathcal{K}(F)} \leq \|G(Nt)\|_{\mathcal{K}}.
\end{equation*}
Then we may repeat the computations above and apply Nussbaum's Theorem again to get
\begin{equation*}
  r_{ess}(G|_F(t)) \leq r_{ess}(G(t)).
\end{equation*}

\end{proof}

\begin{proof}[Proof of Corollary \ref{firstcorollary}.]
From the definitions of $\Theta_*(t)$ and $\Theta_F(t)$ we have
\begin{equation*}
 \sup_{(x_0,\xi_0,b_0)\in \mathcal{A}}|b(x_0,\xi_0,b_0;t)| = \max\{\Theta_*(t), \Theta_F(t)\}.
\end{equation*}
Thus $\mu = \max\{\mu_{3*},\mu_{3F}\}$ where $\mu$ is the Lyapunov-type exponent defined in Theorem \ref{vishik}.  By Theorem \ref{vishik} and Theorem \ref{3Dmainthm} we have
\begin{equation*}
 r_{ess}(G(t)) = e^{\mu t} = \max \{e^{\mu_{3*}t}, e^{\mu_{3F}t}\} \leq \max\{r_{ess}(G_F(t)), r_{ess}(G(t)\mid_{\overline{\text{Im}B}})\}.
\end{equation*}
Then by Proposition \ref{ressprop}
\begin{equation*}
 r_{ess}(G(t)) = \max\{r_{ess}(G_F(t)), r_{ess}(G(t)\mid_{\overline{\text{Im}B}})\}.
\end{equation*}

\end{proof}

\begin{proof}[Proof of Corollary \ref{secondcorollary}]
If we assume $\supp (\omega) = \T^3$, then $\mu = \mu_{3*}$, where $\mu$ is the Lyapunov-type exponent from Theorem \ref{vishik}.  Then by Theorem \ref{vishik} and Theorem \ref{3Dmainthm} $r_{ess}(G(t)) = e^{\mu t} \leq r_{ess}(G(t)|_{\overline{\text{Im}B}})$.  Hence $r_{ess}(G(t)) =  r_{ess}(G(t)|_{\overline{\text{Im}B}})$.

\end{proof}

\section{Main Theorems for 2-dimensional flows}
\label{mainthms2D}
In this section we prove the main theorem for 2-dimensional flows, Theorem \ref{2Dmainthm} below.  Here our vector field $u$ is two-dimensional smooth solution to steady Euler's equation (SE) with scalar vorticity $\omega := \curl u$.  The set of admissible initial conditions for (BAS) in 2-dimensions are the same:
\begin{equation*}
 \mathcal{A}:= \{(x_0,\xi_0,b_0)\in \T^2 \times \R^2 \times \R^2 |\ \xi_0 \perp b_0 ,\ |\xi_0|=|b_0|=1 \}.  
\end{equation*}   We begin with two propositions similar to Proposition \ref{lowerbnd} from Section \ref{mainthms3D}.
\begin{prop}
\label{2DlowerbndIm}
Fix $T>0$ and define $\Theta_*(t)$ by
\begin{equation*}
\Theta_*(t) =  \sup_{\substack {(x_0,\xi_0,b_0)\in \mathcal{A}\\x_0 \in \supp (\nabla \omega)}} |b(x_0, \xi_0, b_0; t)|.
\end{equation*}
Then for any $\varepsilon >0$ and $t\in [0,T]$ we have
\begin{equation*}
\|G_{\varepsilon}^s(t)\|_{\mathcal{L}(\overline{\text{Im}B}, L^2_{sol})} + O(\sqrt{\varepsilon}) \geq \Theta_*(t),
\end{equation*}
where the constant in O is uniform for $t\in[0,T]$.
\end{prop}
\begin{prop}
\label{2DlowerbndF}
Fix $T>0$.  
\begin{itemize}
\item[(i)]If we define $\tilde{\Theta}_F(t)$ by
\begin{equation*}
\tilde{\Theta}_F(t) :=  \sup_{\substack {(x_0,\xi_0,b_0)\in \mathcal{A}\\x_0 \notin \supp \nabla \omega}} |b(x_0, \xi_0, b_0; t)|.
\end{equation*} 
Then for any $\varepsilon >0$ and $t\in [0,T]$ we have
\begin{equation*}
\|G_{\varepsilon}^s(t)\|_{\mathcal{F}} + O(\sqrt{\varepsilon}) \geq \tilde{\Theta}_F(t),
\end{equation*}
where the constant in O is uniform for $t\in[0,T]$.
\item[(ii)]If we define $\overline{\Theta}_F(t)$ by
\begin{equation*}
\overline{\Theta}_F(t) :=  \sup_{\substack {\{x_0\in \T^2 |\ |\nabla\omega(x_0)|>0\}\\|b_0|=1\\
    b_0\perp \nabla\omega(x_0)}} |b(x_0, \nabla\omega(x_0), b_0; t)|.
\end{equation*}   Then for any $\varepsilon >0$ and $t\in[0,T]$ we have
\begin{equation*}
\|G_{\varepsilon}^s(t)\|_{\mathcal{F}} + O(\sqrt{\varepsilon}) \geq \overline{\Theta}_F(t),
\end{equation*}
where $\|\cdot\|_{\mathcal{F}}$ is the seminorm from Definition \ref{F-norm} and the constant in O is uniform for $t\in[0,T]$.
\end{itemize}
\end{prop}
The proofs of these propositions are very similar to the proof of Proposition \ref{lowerbnd}.  First we approximate the evolution of our general 2-dimensional fast oscillating perturbations.  Consider the vector field $\phi_{\delta}\in C^{\infty}(\T^2)$ defined by
\begin{equation}
\label{2Dphidelta}
\phi_{\delta}(x):= \delta \nabla^{\perp}(h_0(x)e^{i\xi_0\cdot x/\delta}),
\end{equation}
where $\delta^{-1} \in \Z_+, \delta < 1, \xi_0 \in \Z^2$ and $h_0 \in C^\infty(\T^2)$ is an arbitrary smooth scalar function.  If we consider $\phi_{\delta}$ as a 3-dimensional planar vector field on $\T^3$, then
\begin{equation}
\label{3Dphi}
 \phi_{\delta} = \delta \nabla \times \left(\frac{i\xi_0 \times \xi_0^{\perp}}{|\xi_0|^2}h_0(x)e^{ix\cdot\xi_0/\delta}\right).
\end{equation}
Thus, by Lemma \ref{imagepsilemma} and Remark \ref{Atobremark} from Section \ref{3Dclassifysection} we have
\begin{equation}
\|\text{op}_{\varepsilon}[a_0]\circ\mathfrak{g}_u(t) \phi_{\delta}\|_{L^2} 
\label{imageofphi}=\ \|h_0(g^{-t}\cdot)b(\cdot,(g^{-t}_*(\cdot))^*\xi_0, \xi_0^{\perp},t )\|_{L^2} + O(\delta).
\end{equation}
We also remark that from the proof of Proposition \ref{lowerbnd} we have the 2-dimensional estimate:
\begin{equation}
\label{GstoG}
\|G_\varepsilon^s(t) - \opeps[a_0] \circ \mathfrak{g}_u(t)\|_{\mathcal{L}(L^2)}
=  O(\sqrt{\varepsilon}).
\end{equation}
\begin{proof}[Proof of Proposition \ref{2DlowerbndIm}]
Let $x_0\in \T^2$, $\xi_0\in \Z^2$ such that $(\xi_0^{\perp}, \nabla \omega(x_0)) \neq 0$.  We can choose $h_0\in C^{\infty}(\T^2)$ supported such that there is some constant $c_0$ where $|(\xi_0^{\perp}, \nabla \omega (x))| > c_0$ for all $x \in \supp (h_0)$.  We will call any function $h_0$ that satisfies these properties, \textit{localized at $x_0$}.  For $\delta^{-1}\in \Z_+$, let $\phi_{\delta}:= -i\delta\nabla^{\perp}(h_0e^{ix\cdot\xi_0/\delta})$.  Then from Lemma \ref{phideltainIm}, $\phi_{\delta}$ is approximately in the image of $B$.  More specifically, there is some remainder $r_{\delta}$ such that  $\|r_{\delta}\|_{L^2_{sol}} = O(\delta)$ and $\phi_{\delta} + r_{\delta} \in \overline{\text{Im}B}$.
Take the limit of the estimate \ref{imageofphi} as $\delta \to 0$, to get
\begin{equation*}
\|\text{op}_\varepsilon[a_0] \circ \mathfrak{g}_u(t)\|_{\mathcal{L}(\overline{\text{Im}B},L^2)}
\geq \sup_{\substack{x_0\in\T^2, \xi_0\in \Z^2\\(\xi_0^{\perp}, \nabla \omega(x_0)) \neq 0\\h_0 \text{ localized at }x_0 }}\frac{\|h_0(g^{-t}\cdot)b(g^{-t}\cdot, \xi_0, \xi_0^{\perp};t)\|_{L^2_{sol}}}{\|h_0\xi^{\perp}_0\|_{L^2_{sol}}}.
\end{equation*}
Then by an argument similar to that for line (\ref{limitA}), we have
\begin{equation}
\label{superone}
\|\text{op}_\varepsilon[a_0] \circ \mathfrak{g}_u(t)\|_{\mathcal{L}(\overline{\text{Im}B},L^2)} \geq  \sup_{\substack{(x_0,\xi_0,b_0) \in \mathcal{A}\\ (\xi_0^{\perp},\nabla \omega(x_0)) \neq 0}}
|b(x_0,\xi_0,b_0;t)|.
\end{equation}
Take the closure of the condition $(\xi_0^{\perp},\nabla \omega(x_0))\neq 0$ on the supremum in line (\ref{superone}) and, since $b(x_0,\xi_0,\xi_0^{\perp};t)$ depends continuously on the initial conditions, we have
\begin{equation*}
\|\text{op}_\varepsilon[a_0] \circ \mathfrak{g}_u(t)\|_{\mathcal{L}(\overline{\text{Im}B}, L^2_{sol})}
\geq \sup_{\substack{(x_0,\xi_0, b_0) \in \mathcal{A} \\ x_0\in \supp \nabla\omega}}
|b(x_0,\xi_0,b_0;t)| =:\Theta_*(t).
\end{equation*}
Hence, from (\ref{GstoG}), we have $\|G_\varepsilon^s(t)\|_{\mathcal{L}(\overline{\text{Im}B}, L^2_{sol})} + O(\sqrt{\varepsilon}) \geq \Theta_*(t)$.  This concludes the proof of Proposition \ref{2DlowerbndIm}.
\end{proof}

\begin{proof}[Proof of Proposition \ref{2DlowerbndF} (i)]
Let $h_0\in C^{\infty}$ such that $\nabla\omega(x) = 0$ for any $x\in\supp h_0$.  Now let $\delta^{-1}\in \Z_+$ and choose any $\xi_0\in \Z^2$ and consider the resulting fast oscillating vector field, $\phi_{\delta}:= -i\delta\nabla^{\perp}(h_0e^{ix\cdot\xi_0/\delta})$.  Just as in the proof of Lemma \ref{phideltainKer} we consider the operator $T = \curl B$ defined by
\begin{equation*}
 Tv:= v\cdot \nabla\omega \hspace{.5 cm}v \in (C^{\infty}(\T^2))^2.
\end{equation*}
Clearly, $\phi_{\delta}\in \text{Ker}T = \text{Ker}B$.  Hence, recalling the expansion from line (\ref{phiexpanded}) we have
\begin{equation}
\label{Fnormphi}
 \|\phi_{\delta}\|_F = \|\phi_{\delta}\|_{L^2} = \|h_0\xi_0^{\perp}\|_{L^2} + O(\delta).
\end{equation}
The vector field $\nabla\omega$ evolves like a covector along the flow $g^t$ and we have
\begin{equation}
 \label{gradvortevolution}
\nabla \omega(g^tx_0) = (g^{-t}_*(x_0))^*\nabla \omega(x_0).
\end{equation}   
It follows that $\nabla\omega \equiv 0$ on $\supp (h_0\circ g^{-t})$ and
\begin{equation*}
h_0(g^{-t}x)b(g^{-t}x, \xi_0, \xi_0^{\perp};t)e^{ig^{-t}x\cdot\xi_0/\delta} \in \text{Ker}T.
\end{equation*}
Hence, from the estimate (\ref{imageofphi}) we have
 \begin{align}
\nonumber
\|\opeps[a_0]\circ \mathfrak{g}^t_u \phi_{\delta}\|_F =& \|h_0(g^{-t}x)b(g^{-t}x, \xi_0, \xi_0^{\perp};t)e^{ig^{-t}x\cdot\xi_0/\delta}\|_F + O(\delta)\\
\label{Fnormimage}
&= \|h_0(g^{-t}x)b(g^{-t}x, \xi_0, \xi_0^{\perp};t)e^{ig^{-t}x\cdot\xi_0/\delta}\|_{L^2} + O(\delta).
\end{align}
Consider (\ref{Fnormphi}) and (\ref{Fnormimage}) and take the limit as $\delta \to 0$ to estimate a lower bound for the $\mathcal{F}-$seminorm of $\opeps[a_0]\circ \mathfrak{g}^t_u$:
\begin{equation*}
 \|\opeps[a_0]\circ \mathfrak{g}^t_u \|_{\mathcal{F}} \geq \sup_{\substack{\xi\in\Z^2, x\in\T^2\\ \supp(h_0)\subset \{x: \nabla\omega(x)=0\}}}\frac{\|h_0(g^{-t}x)b(g^{-t}x, \xi_0, \xi_0^{\perp};t)e^{ig^{-t}x\cdot\xi_0/\delta}\|_{L^2}}{\|h_0\xi_0^{\perp}\|_{L^2}}.
\end{equation*}
Again we use an argument similar to that for line (\ref{limitA}) to simplify the supremum on the RHS.  Here we must also note that if $\supp (h_0)\subset \{x: \nabla\omega(x)=0\}$, then $\supp (h_0\circ g^{-t})\subset \{x: \nabla\omega(x)=0\}$, so we may take the supremum over $x_0 = g^{-t}x \in \T^2 \setminus \supp\nabla\omega$ to get
\begin{equation*}
\|\opeps[a_0]\circ \mathfrak{g}^t_u \|_{\mathcal{F}} \geq \sup_{\substack{(x_0,\xi_0,b_0) \in \mathcal{A}\\ x_0 \notin \supp \nabla\omega(x_0)}}|b(x_0, \xi_0, b_0;t)|=:\tilde{\Theta}_F(t).
\end{equation*}
Therefore, from the estimate (\ref{Fnormestimate}) we have
\begin{equation*}
 \|G_{\varepsilon}^s(t)\|_{\mathcal{F}} + O(\sqrt{\varepsilon}) \geq \tilde{\Theta}_F(t).
\end{equation*}

\end{proof}

\begin{proof}[Proof of Proposition \ref{2DlowerbndF} (ii)]
Let $x_0 \in \T^n$ such that $\nabla\omega(x_0)\neq 0$ and define $\xi_0 := \frac{\nabla\omega(x_0)}{|\nabla\omega(x_0)|}$.  Let $h_0\in C^{\infty}(\T^n)$ be supported on $B_1(0)$, the ball of radius $1$ centered at $0$ such that $h_0(0)=1$.  For $0<\zeta<<1$ define $h_{\zeta}$ by
\begin{equation*}
h_{\zeta}(x):= h_0\bigl(\frac{x-x_0}{\zeta}\bigr)
\end{equation*}
and let $\delta^{-1}\in \Z_+$.  For any $x \in [0,1)\times [0,1)$ define
\begin{equation*}
\phi_{\zeta,\delta}(x):= -i\delta \nabla^{\perp}(h_{\zeta}(x)e^{ix\cdot\xi_0/\delta}),
\end{equation*}
and extend $\phi_{\zeta,\delta}$ periodically.
It follows from Lemma \ref{phideltainKer} and the expansion (\ref{phiexpanded}) of $\phi_{\zeta,\delta}$ that
\begin{equation}
\label{Fnormphizeta}
\|\phi_{\zeta,\delta}\|_F = \|\phi_{\zeta,\delta}\|_{L^2} + O(\zeta) + O(\delta) = \|h_{\zeta}\xi_0^{\perp}\|_{L^2} + O(\zeta) + O(\delta).
\end{equation}
Now we must estimate $\|\opeps[a_0]\circ \mathfrak{g}^t_u \phi_{\zeta,\delta}\|_F$ for our fixed time $t>0$.  The approach is similar to the proof of Lemma \ref{phideltainKer}.  Here we we will also use that for the operator $T$ defined by $Tv = \curl Bv = v\cdot \nabla\omega$, we have $\text{Ker}B = \text{Ker}T$.  Lemma \ref{imagephilemma} gives that
\begin{equation}
\label{step3}
 \bigl(\opeps[a_0]\circ \mathfrak{g}^t_u \phi_{\zeta,\delta}\bigr)(x) = h_{\zeta}(g^{-t}x)b(g^{-t}x,\xi_0,\xi_0^{\perp};t)e^{ig^{-t}x\cdot\xi_0/\delta} + r_{\delta}(x),
\end{equation}
where $\|r_{\delta}\|_{L^2} = O(\delta)$.  

Let $y\in \supp (h_{\zeta}\circ g^{-t})$ and $b(y):=b(g^{-t}y,\xi_0, \xi_0^{\perp};t)$  (notice that the parameters $\xi_0$ and $t$ are fixed).
 Then we have
\begin{equation*}
 b(y) = b(y) - \frac{(b(y),\nabla\omega(y))}{|\nabla\omega(y)|^2}\nabla\omega(y) + \frac{(b(y),\nabla\omega(y))}{|\nabla\omega(y)|^2}\nabla\omega(y).
\end{equation*}

We will now demonstrate that 
\begin{equation*}
\frac{(b(y),\nabla\omega(y))}{|\nabla\omega(y)|} = O(\zeta).  
\end{equation*}
Let $y_0:= g^{-t}y$, hence $y_0\in \supp(h_{\zeta}) \subset B_{\zeta}(x_0)$.   Then we may estimate 
\begin{align}
\nonumber
 |(b(y),\nabla\omega(y))| \leq& |(b(y), \nabla\omega(g^{t}y_0)) -(b(y),\nabla\omega(g^tx_0))| + |(b(y),\nabla\omega(g^tx_0))| \\
 \label{step1} \leq&\ \zeta K\|b(\cdot)\|_{L^{\infty}(\T^2)} + |(b(y),\nabla\omega(g^tx_0))|,
\end{align}
where $K$ is the Lipschitz norm of $\nabla\omega \circ g^t$ on $\T^2$.  Recall that we defined $\xi_0 := \frac{\nabla\omega(x_0)}{|\nabla\omega(x_0)|}$, so from the construction of (BAS), see equation (\ref{bperpxi}), we have
\begin{equation*}
(b(g^tx_0),\nabla\omega(g^tx_0)) = (b(x_0,\xi_0,\xi_0^{\perp};t),(g^{-t}_*(x_0))^*\xi_0 ) = 0.
 \end{equation*}
It follows that
\begin{equation}
\label{step2} |(b(y),\nabla\omega(g^tx_0))| = |(b(g^t y_0),\nabla\omega(g^tx_0))- (b(g^tx_0),\nabla\omega(g^tx_0))|
 \leq \zeta L\|\nabla\omega\|_{L^{\infty}(\T^2)},
\end{equation}
where $L$ is the Lipschitz norm of the function $x \to b(g^tx)$.  We may assume $\zeta << 1$, which implies $|\nabla\omega(y)| \geq |\nabla\omega(g^tx_0)| - \zeta K >0$.  Thus from (\ref{step1}) and (\ref{step2}) we have
\begin{equation*}
 \frac{|(b(y),\nabla\omega(y))|}{|\nabla\omega(y)|} \leq\ \frac{\zeta K\|b(\cdot)\|_{L^{\infty}(\T^2)} + \zeta L\|\nabla\omega\|_{L^{\infty}(\T^2)}}{|\nabla\omega(g^tx_0)|-\zeta K} = O(\zeta),
\end{equation*}
where $O(\zeta)$ is uniform in $x$ and independent of $\delta$.

For any $x\in \supp (h_{\zeta}\circ g^{-t})$ we define $\eta(x)$ by
\begin{equation*}
\eta(y):= b(x) - \frac{(b(x),\nabla\omega(x))}{|\nabla\omega(x)|},
\end{equation*}
Then $h_{\zeta}(g^{-t}\cdot)\eta(\cdot) \in C^{\infty}(\T^n)$ and
\begin{equation*}
h_{\zeta}(g^{-t} x)b(x) = h_{\zeta}\circ g^{-t}\eta(x) + O(\zeta)\frac{h_{\zeta}(g^{-t} x)\nabla\omega(x)}{|\nabla\omega(x)|},
\end{equation*}
where $O(\zeta)$ is uniform in $x$ and is independent of $\delta$.  From the definition of $\eta$, it is clear that
\begin{equation*}
 T\bigl(h_{\zeta}(g^{-t}\cdot)\eta e^{ig^{-t}(\cdot)\cdot\xi_0}\bigr)(x) = h_{\zeta}(g^{-t}x)e^{ig^{-t}x\cdot\xi_0}\bigl(\eta\cdot\nabla\omega\bigr)(x) \equiv 0.
\end{equation*}  
Hence $h_{\zeta}(g^{-t}\cdot)\eta e^{ig^{-t}(\cdot)\cdot\xi_0}\in \text{Ker}B$ and, because $\|h_{\zeta}\|_{L^2} = \zeta\|h_0\|_{L^2}$ on $\T^2$, we have
\begin{equation*}
\|h_{\zeta}(g^{-t} \cdot)b e^{ig^{-t}(\cdot)\cdot\xi_0/\delta}\|_F
 =\ \|h_{\zeta}(g^{-t} \cdot)b e^{ig^{-t}(\cdot)\cdot\xi_0/\delta}\|_{L^2} + O(\zeta^2)
\end{equation*}
Then from (\ref{step3}) we have
\begin{equation}
\label{opphiF}
 \|\opeps[a_0]\circ \mathfrak{g}^t_u \phi_{\zeta,\delta}\|_{F} 
=\ \|\opeps[a_0]\circ \mathfrak{g}^t_u \phi_{\zeta, \delta}\|_{L^2} + O(\zeta^2) + O(\delta),
\end{equation} 
where the $O(\zeta^2)$ does not depend on $\delta$.  

Consider the quotient (\ref{opphiF}) over (\ref{Fnormphizeta}) and take the limit as $\delta \to 0$ to get,
\begin{equation}
\label{fatmess}
\|\text{op}_\varepsilon[a_0] \circ \mathfrak{g}_u(t)\|_{\mathcal{F}} + O(\zeta^2)
\geq \sup_{\substack{x,\ x_0\in\T^2\\|\nabla\omega(x_0)|>0\\ \xi_0 = \nabla\omega(x_0)/|\nabla\omega(x_0)|}}
\frac{\|h_{\zeta}(g^{-t}x)b(g^{-t}x,\xi_0,\xi_0^{\perp};t)\|_{L^2}} {\|h_{\zeta}\xi_0^{\perp}\|_{L^2}}.
\end{equation}
For any value of $0<\zeta<1$, $h_{\zeta}(x_0) = 1$, so for fixed $|\xi_0|=1$ we have
\begin{equation*}
\lim_{\zeta\to 0}\frac{\|h_{\zeta}(g^{-t}x)b(g^{-t}x,\xi_0,\xi_0^{\perp};t)\|_{L^2}} {\|h_{\zeta}\xi_0^{\perp}\|_{L^2}} = |b(x_0, \xi_0,\xi_0^{\perp};t)|.
\end{equation*}
Hence, we can take the limit as $\zeta\to 0$ of (\ref{fatmess}) (and use the fact that $b$ is homogeneous of degree $0$ in $\xi_0$) to get
\begin{equation*}
\|\text{op}_\varepsilon[a_0] \circ \mathfrak{g}_u(t)\|_{\mathcal{F}} \geq \sup_{\substack{|\nabla\omega(x_0)|>0,\ |b_0|=1\\
    b_0\perp \nabla\omega(x_0)}}|b(x_0,\nabla\omega(x_0),b_0;t)| =: \overline{\Theta}_F(t).
\end{equation*}
From (\ref{GstoG}) we have
\begin{equation*}
\|G_\varepsilon^s(t) - \opeps[a_0] \circ \mathfrak{g}_u(t) \|_{\mathcal{F}} = O(\sqrt{\varepsilon}).
\end{equation*}
Therefore, $\|G_\varepsilon^s(t)\|_{\mathcal{F}} + O(\sqrt{\varepsilon}) \geq \overline{\Theta}_F(t)$.
\end{proof}

\begin{definition}
Let $\Theta_F(t):= \max\{\tilde{\Theta}_F(t), \overline{\Theta}_F(t)\}$ and define $\mu_{2*}, \mu_{2F} \in\R$ by
\begin{align*}
\mu_{2*} &= lim_{t\to\infty}\frac{1}{t}\log\Theta_*(t),\\
\mu_{2F} &= lim_{t\to\infty}\frac{1}{t}\log \Theta_F(t).
\end{align*} The existence of both limits follows from Remark \ref{limitexist}.
\end{definition} 

\begin{theorem}
\label{2Dmainthm}
For 2-dimensional flows, we have the following lower bound for the essential spectral radius of our evolution operator restricted to $\overline{\text{Im}B}$:
\begin{equation*}
e^{\mu_{2*}t} \leq r_{ess}(G(t)|_{\overline{\text{Im}B}}).
\end{equation*}
And for 2-dimensional flows we have another lower bound for the essential spectral radius of the evolution operator acting on the factor space:
\begin{equation*}
e^{\mu_{2F}t} \leq r_{ess}(G_F(t)),
\end{equation*}
where $G_F(t)$ denotes $G(t)$ on the factor space.
\end{theorem}

\begin{proof}
 The proof for Theorem \ref{2Dmainthm} is the same as that for Theorem \ref{3Dmainthm} except that we will use the 2-dimensional propositions from the current section instead of Proposition \ref{lowerbnd} (see Remark \ref{mainthmrmk} following the proof of Theorem \ref{3Dmainthm}).  To prove $e^{\mu_{2*}t} \leq r_{ess}(G(t)|_{\overline{\text{Im}B}})$ replace Proposition \ref{lowerbnd} with Proposition \ref{2DlowerbndIm} in the proof of Theorem \ref{3Dmainthm} for $\overline{\text{Im}B}$.  For the factor space estimate, notice that Proposition \ref{2DlowerbndF} implies
\begin{equation}
\label{fsestimate}
\|G_{\varepsilon}^s(t)\|_{\mathcal{F}} + O(\sqrt{\varepsilon}) \geq \Theta_F(t),
\end{equation}
 where $\Theta_F(t) := \max\{\tilde{\Theta}_F(t), \overline{\Theta}_F(t)\}$.  To prove $e^{\mu_{2F}t} \leq r_{ess}(G_F(t))$, replace Proposition \ref{lowerbnd} with the estimate (\ref{fsestimate}) above in the proof of Theorem \ref{3Dmainthm} for the factor space.
\end{proof}

\begin{corollary}For flows in 2D
 \begin{equation*}
r_{ess}(G(t)) = \max\{r_{ess}(G_F(t)), r_{ess}(G(t)\mid_{\overline{\text{Im}B}})\}.
\end{equation*}
\end{corollary}

\begin{proof}
By Proposition \ref{ressprop} we have
\begin{equation*}
\max\{r_{ess}(G_F(t)), r_{ess}(G(t)\mid_{\overline{\text{Im}B}})\}\leq r_{ess}(G(t)).
\end{equation*}
For the other inequality, notice
\begin{equation*}
\sup_{(x_0,\xi_0,b_0) \in \mathcal{A}}|b(x_0,\xi_0,b_0;t)| = \max\{\Theta_*(t), \tilde{\Theta}_F(t)\}.
\end{equation*}
Hence
\begin{equation}
\label{blarb}
\lim_{t\to\infty}\frac{1}{t} \log \sup_{(x_0,\xi_0,b_0) \in \mathcal{A}}|b(x_0,\xi_0,b_0;t)| \leq \lim_{t\to\infty}\frac{1}{t} \log \max \{\Theta_*(t), \tilde{\Theta}_F(t)\}.
\end{equation}
The LHS of (\ref{blarb}) is the Lyapunov-type exponent $\mu$ from Theorem \ref{vishik}, so we have
\begin{equation*}
r_{ess}(G(t)) = e^{\mu t} \leq \max\{e^{\mu_{2*}t}, e^{\mu_{2F}t}\} \leq \max\{r_{ess}(G_F(t)), r_{ess}(G(t)\mid_{\overline{\text{Im}B}})\}.
\end{equation*}
\end{proof}

\section{Example: Hyperbolic Stagnation Point}
\label{2Dhspsection}

A point $x_s \in \T^n$ is a hyperbolic stagnation point of the flow if $\pupx (x_s)$ does not have any purely imaginary eigenvalues.  In \cite{fv2} Friedlander and Vishik demonstrate that for any flow with a hyperbolic stagnation point, there is instability in the essential spectrum.  Moreover, any instability in the essential spectrum for a 2-dimensional flow is caused by a hyperbolic stagnation point, see \cite{fsv}.  Here we see that for two-dimensional flows where the hyperbolic stagnation point $x_s$ is in the support of the gradient of vorticity, this instability is caused by perturbations in $\overline{\text{Im}B}$ as well as by perturbations in the factor space.  At the end of this section we use this fact to demonstrate that 3-dimensional planar flows with a hyperbolic stagnation point also have instability in the factor space - regardless of whether or not the stagnation point is in $\supp(\omega)$.

Suppose the two-dimensional steady flow $u$ has a hyperbolic stagnation point, $x_s$ and that $x_s \in \supp \nabla\omega$.  In \cite{fv2}, the authors demonstrate that (BAS) has a simple solution at the hyperbolic stagnation point with the following argument.  A straightforward computation shows that $\pupx(x_s)$ is a symmetric matrix whenever $x_s$ is a hyperbolic stagnation point of an inviscid, incompressible flow - in 2 or 3 dimensions.  Since we are in 2-dimensional space, it follows that $\pupx (x_s)$ has two real eigenvalues and the divergence free condition gives us that the sum of these eigenvalues is $0$.  Let $\lambda$ and $-\lambda$ be the eigenvalues of $\pupx (x_s)$ associated with the eigenvectors $a_+$ and $a_-$, respectively.  Then at a hyperbolic stagnation point we always have a solution to (BAS) of this form:
\begin{align*}
&x(t)= x_s\\
&\xi(t) = a_+ e^{-\lambda t}\\
&b(t) = a_- e^{\lambda t}.
\end{align*}

From the definition of $\Theta_*(t)$ in Proposition \ref{2DlowerbndIm} we see that $\Theta_*(t) \geq e^{\lambda t}$.  Hence $\mu_{2*} \geq \lambda >0$ and we have exponential stretching of perturbations in $\overline{\text{Im}B}$.  

In order to use Proposition \ref{2DlowerbndF}  to demonstrate that there is exponential growth in the factor space, we must find a solution to (BAS) that is close enough to the solution at the hyperbolic stagnation point above to exhibit exponential stretching, and this solution must grow in the factor space norm.  Locally, two perpendicular flow lines pass through the hyperbolic stagnation point.  Along one, the stable flow line, the fluid moves towards the stagnation point.  Along the other, the unstable flow line, the fluid moves away from the stagnation point.   Choose a point $x_0$ on the stable flow line, so that $x(t)$ will flow into the stagnation point.  We remark that $\nabla \omega (x_0)$ moves like a covector along the flow, thus $\xi(t) = \nabla\omega(g^tx_0)$ satisfies the $\xi$-equation of (BAS).  Let $\xi_0 = \frac{\nabla\omega(x_0)}{|\nabla\omega(x_0)|}$ and let $b_0=\xi_0^{\perp}$.  The resulting solution to (BAS) flows into the hyperbolic stagnation point solution above in the sense that as $s \rightarrow \infty$
\begin{equation*}
g^sx_0 \rightarrow x_s,\hspace{.3 cm} \frac{\xi(s)}{|\xi(s)|} \rightarrow a_+ \hspace{.2 cm} \text{and}\hspace{.2 cm} \frac{b(s)}{|b(s)|}\rightarrow a_-.
 \end{equation*}
 Hence, the solution $b(x_0,\xi_0,b_0;t)$ approaches the solution $b(x_s,a_+,a_-;t)$ as we choose values for $x_0$ closer to $x_s$.  This implies that $\overline{\Theta_F}(t) \geq e^{\lambda t}$ and by Proposition \ref{2DlowerbndF} we have exponential stretching in the factor space.
 


Finally, we use Theorem \ref{2Dmainthm} for a 3-dimensional planar flow to detect instability in the factor space that Theorem \ref{3Dmainthm} cannot detect.  Consider the planar 3-dimensional steady flow given by
\begin{equation*}
u_1(x):= \sin x_1 \cos x_2 \hspace{.5 cm} u_2(x):= -\cos x_1 \sin x_2 \hspace{.5 cm} u_3(x) = 0.
\end{equation*}
The point $x_s = (0,0,0)$ is a hyperbolic stagnation point since
\begin{equation*}
\pupx(x_s) = \left( \begin{array}{ccc}
1 & 0 & 0\\
0 & -1 & 0 \\
0 & 0 & 0 \end{array} \right).
\end{equation*}
Clearly, the eigenvalues of $\tfrac{\partial u}{\partial x}(x_s)$ are $\pm 1$ with corresponding eigenvectors in the first two coordinate directions.  This implies that we have some linear instability in the essential spectrum.  Theorem \ref{3Dmainthm} (and Corollary \ref{secondcorollary}) gives us that this instability corresponds to a perturbation in $\overline{\text{Im}B}$.  However, since $\supp(\omega) = \T^3$ in this example, we cannot compute $\mu_{3F}$ for this flow and Theorem \ref{3Dmainthm} tells us nothing about the factor space.   

For 3-dimensional planar flows we can use Theorem \ref{2Dmainthm} to compute a lower bound for $r_{ess}(G_F(t))$.  Notice that if $u$ is planar, then any vector field in $\overline{\text{Im}B}$ is co-planar.  It follows that the 2-dimensional factor space norm of an operator is less than or equal to the 3-dimensional factor space norm of the same operator.   Also, the essential spectrum of $G(t)$ in 3-dimensions contains the essential spectrum of $G(t)$ in 2-dimensions.  Thus, the essential spectral radius of $G_F(t)$ in 2-dimensions is less than or equal to the essential spectral radius of $G_F(t)$ in 3-dimensions and $e^{\mu_{2F}t} \leq r_{ess}(G_F(t))$ in 3-dimensions.  Since $\mu_{2F} > 0$ in this example, we have instability in the factor space.

\newpage
\nocite{*}
\bibliographystyle{plain}
	\bibliography{lwerbnd}
\end{document}